\documentclass[11pt]{article}

\usepackage[T1]{fontenc}
\usepackage[utf8]{inputenc}
\usepackage{lmodern}        
\usepackage{iftex}          
\ifPDFTeX                   
  \usepackage[activate={true,nocompatibility},final]{microtype}
\else                       
  \usepackage[protrusion=true,final]{microtype}
\fi

\usepackage{times}

\usepackage[margin=1in]{geometry}

\usepackage{authblk}                    

\usepackage{amsmath}
\usepackage{amssymb,amsfonts}
\usepackage{mathtools}      
\usepackage{bm}             
\usepackage{nicefrac}
\allowdisplaybreaks         

\usepackage{amsthm}
\usepackage{thm-restate}    

\usepackage{graphicx}
\graphicspath{{figures/}{./}}                       
\usepackage{booktabs}                               
\usepackage{multirow,makecell,array}
\usepackage[font=small,labelfont=bf]{caption}       
\usepackage{subcaption}                             
\usepackage{algorithm}
\usepackage{algpseudocode}                          
\makeatletter
\renewcommand{\theHALG@line}{\thealgorithm.\arabic{ALG@line}}
\makeatother

\usepackage{enumitem}
\setlist[itemize]{leftmargin=2.2em,itemsep=2pt,topsep=2pt}
\setlist[enumerate]{leftmargin=2.2em,itemsep=2pt,topsep=2pt}

\usepackage{xcolor}
\definecolor{LinkColor}{rgb}{0.10,0.40,0.75}        
\definecolor{CiteColor}{rgb}{0.70,0.25,0.20}        
\definecolor{UrlColor} {rgb}{0.20,0.50,0.50}        
\definecolor{TodoColor}{rgb}{0.80,0.30,0.10}        

\usepackage{tikz}
\usetikzlibrary{positioning,calc,arrows.meta}

\usepackage[round,sort&compress]{natbib}

\usepackage{url}
\usepackage{hyperref}
\hypersetup{colorlinks=true,linkcolor=red!70!black,linktocpage=false,citebordercolor=blue!70!black,citecolor=blue!70!black,anchorcolor=blue!70!black}
\usepackage{bookmark}                               
\usepackage[capitalise,nameinlink,noabbrev,sort&compress]{cleveref}  

\numberwithin{equation}{section}

\newif\ifdraft \draftfalse
\ifdraft
  \usepackage{lineno}\linenumbers
  \newcommand{\todo}[1]{\textcolor{TodoColor}{\textbf{[TODO:}~#1\textbf{]}}}
\else
  \newcommand{\todo}[1]{}                            
\fi

\newcommand{\R}{\mathbb{R}}

\newcommand{\eps}{\varepsilon}

\newcommand{\defeq}{\coloneqq}                    

\DeclareMathOperator*{\argmin}{arg\,min}
\newcommand{\Id}{\mathrm{Id}}

\DeclarePairedDelimiter{\abs}{\lvert}{\rvert}
\DeclarePairedDelimiter{\norm}{\lVert}{\rVert}

\DeclarePairedDelimiterX{\inner}[2]{\langle}{\rangle}{#1,#2}   

\newcommand{\Man}{\mathcal{M}}
\newcommand{\Hyp}{\mathbb{H}}
\newcommand{\Exp}{\operatorname{Exp}}
\newcommand{\Log}{\operatorname{Log}}
\newcommand{\grad}{\operatorname{grad}}
\newcommand{\Sec}{\operatorname{Sec}}
\newcommand{\PT}{\mathsf{P}}
\newcommand{\Qp}{Q_p}
\newcommand{\Fh}{\mathcal{F}^{\mathrm h}_{p,\beta}}
\newcommand{\Fg}{\mathcal{F}^{\mathrm g}_{p}}

\theoremstyle{plain}
\newtheorem{theorem}{Theorem}[section]
\newtheorem{proposition}[theorem]{Proposition}
\newtheorem{lemma}[theorem]{Lemma}

\theoremstyle{definition}
\newtheorem{definition}[theorem]{Definition}

\theoremstyle{remark}

\crefname{theorem}{Theorem}{Theorems}          \Crefname{theorem}{Theorem}{Theorems}
\crefname{proposition}{Proposition}{Propositions}
\Crefname{proposition}{Proposition}{Propositions}
\crefname{lemma}{Lemma}{Lemmas}                \Crefname{lemma}{Lemma}{Lemmas}
\crefname{corollary}{Corollary}{Corollaries}   \Crefname{corollary}{Corollary}{Corollaries}
\crefname{conjecture}{Conjecture}{Conjectures} \Crefname{conjecture}{Conjecture}{Conjectures}
\crefname{fact}{Fact}{Facts}                   \Crefname{fact}{Fact}{Facts}
\crefname{definition}{Definition}{Definitions} \Crefname{definition}{Definition}{Definitions}
\crefname{assumption}{Assumption}{Assumptions} \Crefname{assumption}{Assumption}{Assumptions}
\crefname{example}{Example}{Examples}          \Crefname{example}{Example}{Examples}
\crefname{problem}{Problem}{Problems}          \Crefname{problem}{Problem}{Problems}
\crefname{remark}{Remark}{Remarks}             \Crefname{remark}{Remark}{Remarks}
\crefname{claim}{Claim}{Claims}                \Crefname{claim}{Claim}{Claims}
\crefname{algorithm}{Algorithm}{Algorithms}    \Crefname{algorithm}{Algorithm}{Algorithms}

\title{The Blessing and Curse of Curvature in Higher-Order Optimization:\\
Horospherical versus Geodesic Convexity}
\author{Andi Han}
\affil{School of Mathematics and Statistics, University of Sydney \\ \texttt{andi.han@sydney.edu.au}}
\date{}                                           

\hypersetup{
  pdftitle={The Blessing and Curse of Curvature in Higher-Order Optimization: Horospherical versus Geodesic Convexity},
  pdfauthor={Anonymous authors}
}

\begin{document}
\maketitle

\begin{abstract}
We study deterministic Riemannian \(p\)-th-order oracle
complexity (for \(p\ge2\)) on Hadamard manifolds, under strong horospherical
\((h)\)-convexity and strong geodesic \((g)\)-convexity. The two notions agree
in the Euclidean space. On a curved Hadamard manifold, \(h\)-convexity is a 
stronger notion than \(g\)-convexity and supplies global horospherical information.

Writing the $p$-th-order condition parameter \(Q_p=L_pR^{p-1}/\mu\), we obtain the Euclidean-optimal rate
\(Q_p^{2/(3p+1)}\)  for strongly \(h\)-convex objectives on every
Hadamard manifold, with a matching fixed-curvature lower bound. On hyperbolic
space, the resulting horoball supports enable localization to the curvature scale. 
This replaces the condition parameter \(Q_p\) by $Q_p \min\{1, 4/(\kappa R) \}^{p-1}$, 
subject to a logarithmic localization cost. 
Thus growing negative curvature ($\kappa R \rightarrow \infty$) can further improve 
the optimal Euclidean rate under \(h\)-convexity. For strongly \(g\)-convex objectives, matching upper and lower bounds recover
the same Euclidean exponent when \(\kappa R=O(1)\). In contrast, with growing \(\kappa R\),
 we construct a hard family on the hyperbolic space with
\(Q_p\asymp_p(1+\kappa R)^p\) that requires
\(
\widetilde\Omega_p(Q_p^{1/p})\) queries. This reveals a fundamental separation: the same hyperbolic divergence
that sharpens horoball localization yields an information-theoretic
obstruction for the full \(g\)-convex class.

\end{abstract}

\section{Introduction}
\label{sec:intro}

Higher-order methods can attain the optimal oracle rate for smooth strongly
convex optimization in Euclidean space
\citep{Nesterov2021,KornowskiShamir2020}. On a Riemannian manifold, curvature
changes both the comparison of tangent spaces and the global geometry where information is acquired.
First-order oracle lower bounds show that
this dependence can be unavoidable for the geodesic-convex (\(g\)-convex) class where negative curvature
obstructs full acceleration
\citep{HamiltonMoitra2021,CriscitielloBoumal2022,CriscitielloBoumal2023}. 
A recent paper \citep{CriscitielloKim2025} proves curvature-free first-order acceleration under a 
restricted notion of horospherical convexity (\(h\)-convexity), where curvature appears benign. 

In this paper, we study the optimization problem on Hadamard manifolds 
(a special class of Riemannian manifolds with non-positive curvature)
with access to exact Riemannian \textit{$p$-th-order oracle}. We show that the role of negative curvature also depends on
the notion of convexity, where curvature can improve higher-order
optimization or create an information-theoretic obstruction. 
We compare strong \(h\)-convexity, introduced by
\citet{CriscitielloKim2025}, with the strong \(g\)-convexity.
The distinction vanishes in Euclidean space, 
where strong \(h\)-convexity becomes equivalent to \(g\)-convexity. 
On a Hadamard manifold, strong \(h\)-convexity is a smaller class of
 strong \(g\)-convexity, where 
the \(h\)-convex class supplies additional global structure whose
algorithmic value appears only in curved geometry. Earlier work used
horospherically convex level sets to obtain curvature-independent subgradient
guarantees on Hadamard spaces \citep{LewisLopezAcedoNicolae2024}.

At first order, \citet{CriscitielloKim2025} prove Euclidean-optimal-rate Nesterov
acceleration for both smooth \(h\)-convex and smooth strongly \(h\)-convex
objectives. They also show that gradients localize the minimizer to the
curvature scale in \(O(\log(1+\kappa R))\) queries on hyperbolic space. Their
curvature-improving guarantee combines this localization with a cutting-plane
method, while their smooth accelerated guarantees are curvature-independent.
This distinction matters for the higher-order results in this paper. Fix an oracle order \(p\ge2\), a strong-convexity parameter \(\mu\), a radius
budget \(R\), and a Lipschitz constant \(L_p\) for the \(p\)-th covariant
derivative. The dimensionless condition parameter is
\(\Qp=L_pR^{p-1}/\mu\).
At order one, the strongly
convex condition parameter is \(L_1/\mu\) and contains no radius. At order
\(p\ge2\), the parameter \(L_pR^{p-1}/\mu\) changes when localization replaces
\(R\) by the curvature scale.

\paragraph{Main Results.} 
For strongly \(h\)-convex objectives, we obtain the
optimal \(\Qp^{2/(3p+1)}\) rate on Hadamard manifold  (\cref{thm:h-pth-oracle-upper}), and 
obtain a matching lower bound on hyperbolic space of fixed curvature (\cref{thm:h-pth-oracle-lower}). 
On hyperbolic space of curvature \(-\kappa^2\), we first localize the
minimizer to radius \(O(1/\kappa)\) based on \citep{CriscitielloKim2025}. The key idea is that hyperbolic divergence makes
intersections of suitable horoballs contract. 
This allows to replace \(\Qp\) by the effective
parameter \(\Qp
\min\{1,4/(\kappa R)\}^{p-1}\) (\cref{thm:h-curvature-adaptive-upper}). Thus growing negative curvature can render the
\(h\)-convex problem easier, apart from a logarithmic localization cost. 

For strongly \(g\)-convex objectives, we show in \cref{thm:g-upper}, the same Euclidean exponent remains
optimal when the curvature--radius product is bounded, i.e., $\kappa R = O(1)$. 
We also prove  a matching lower bound on Hadamard manifolds (\cref{thm:g-bounded-curvature-lower}).
In contrast, when \(\kappa R\) grows, we construct
a hard family on hyperbolic space with \(\Qp\asymp_p(1+\kappa R)^p\), 
where every deterministic 
exact Riemannian \(p\)-th-order method requires
\(\Omega_p(\kappa R/\log(2+\kappa R))\) queries, which gives a lower bound
 \(\widetilde\Omega_p(\Qp^{1/p})\) (\cref{thm:g-growing-curvature-lower}). 
 This extends the negative results from first-order \citep{CriscitielloBoumal2022,CriscitielloBoumal2023}
 to higher-order oracles and smoothness, where negative curvature obstructs acceleration.

\paragraph{Relation to prior work.}
Euclidean upper bounds build on accelerated regularized tensor methods
\citep{Nesterov2021}. For strongly convex objectives, the corresponding
condition-number exponent is matched by the oracle lower bound of
\citet{KornowskiShamir2020}. Existing higher-order Riemannian results
primarily address nonconvex stationarity
\citep{AgarwalBoumalBullinsCartis2021,GutmanLobo2026}. Acceleration for
\(g\)-convex optimization has been developed mainly for first-order or
proximal methods, with curvature entering through comparison inequalities and
metric distortion
\citep{ZhangSra2018,KimYang2022,MartinezRubioPokutta2023}. First-order lower
bounds already show that negative curvature can obstruct robust acceleration
for the \(g\)-convex class
\citep{HamiltonMoitra2021,CriscitielloBoumal2022}. Horospherical convexity,
its curvature-independent algorithms, and curvature-scale localization were
introduced by \citet{CriscitielloKim2025} in a first-order setting. To our
knowledge, this work is the \textit{first higher-order oracle-complexity} study
of strongly convex optimization on Hadamard manifolds. It establishes
upper and lower bounds in an exact Riemannian \(p\)-th-order oracle model.
Its further distinction is geometric: the same negative curvature can improve
the complexity of the stronger \(h\)-convex class while obstructing a
curvature-independent accelerated rate for the entire \(g\)-convex class. A
more complete comparison is given in \cref{app:related-work}.

\section{Problem formulation and oracle model}
\label{sec:problem}

We study deterministic higher-order optimization on Hadamard manifolds under two
different notions of convexity. This section gives a preliminary introduction to the 
hadamard geometry and convexity.

\subsection{Hadamard geometry and covariant tensors}
\label{sec:geometry-notation}

Throughout, \(\Man\) is a finite-dimensional Hadamard manifold: it is complete and
simply connected, with sectional curvature \(\Sec_{\Man}\le 0\).  We write \(g_x\)
and \(\norm{\cdot}_x\) for the Riemannian metric and norm on \(T_x\Man\), and
\(d(\cdot, \cdot)\) for the induced distance.  The exponential map
\(\Exp_x:T_x\Man\to\Man\) is a global diffeomorphism. Its inverse is denoted by
\(\Log_x\). Thus any \(x,y\in\Man\) are joined by the unique minimizing
geodesic \(\gamma_{x,y}(t)=\Exp_x(t\Log_x y)\), where \(0\le t\le1\).
A set \(A\subseteq\Man\) is \emph{geodesically convex} if
\(\gamma_{x,y}([0,1])\subseteq A\) for every \(x,y\in A\).
Parallel transport along this geodesic is denoted by
\(\PT_{x\to y}:T_x\Man\to T_y\Man\).  Closed metric balls are written
\(\overline B(x,r)\). They are geodesically convex.

For \(k\ge2\), \(\nabla^k f(x)\) denotes the ordered \(k\)-th covariant
derivative of \(f\) at \(x\).  If \(T\) is a covariant \(k\)-tensor at \(x\),
its operator norm is
\(\norm{T}_{x,\mathrm{op}}\defeq
\sup_{\norm{v_1}_x=\cdots=\norm{v_k}_x=1}
\abs{T[v_1,\ldots,v_k]}\). Transport of covariant tensors is defined
slotwise by
\(\bigl(\PT^{(k)}_{x\to y}T\bigr)[v_1,\ldots,v_k]\defeq
T[\PT_{y\to x}v_1,\ldots,\PT_{y\to x}v_k]\).
We say that the \(p\)-th covariant derivative of \(f\) is
\(L_p\)-Lipschitz on \(A\) if, for all \(x,y\in A\),
\begin{equation}
\label{eq:p-smoothness}
 \norm{\nabla^p f(y)-\PT^{(p)}_{x\to y}\nabla^p f(x)}
      _{y,\mathrm{op}}
 \le L_p d(x,y).
\end{equation}
We refer to \eqref{eq:p-smoothness} as the \(p\)-th-order smoothness
assumption.  It is the intrinsic, Riemannian version of the standard
Euclidean condition used in higher-order oracle complexity
\citep{KornowskiShamir2020}, and it is the condition that controls the
Riemannian Taylor remainder \citep[Proposition~4.12]{GutmanLobo2026}.  It
constrains only the \(p\)-th derivative. No bounds on the lower-order derivatives at
a base point are implicit.

\subsection{Geodesic convexity and horospherical convexity}
\label{sec:convexity-models}

We first recall the definition of strong geodesic convexity, and then introduce 
horospherical convexity, which replaces the tangent-space quadratic lower model by an
intrinsic squared-distance lower model.

\begin{definition}[Strong geodesic convexity
  {\citep{Boumal2023}}]
\label{def:strong-g-convexity}
Let \(A\subseteq\Man\) be geodesically convex.  A differentiable function
\(f:A\to\R\) is \(\mu\)-strongly \(g\)-convex on \(A\) if, for all
\(x,y\in A\),
\begin{equation}
\label{eq:strong-g-convexity}
 f(y)\ge
 f(x)+\inner{\grad f(x)}{\Log_x y}_{x}
       +\frac{\mu}{2}d(x,y)^2 .
\end{equation}
For \(C^2\) functions, identify the covariant Hessian with its associated
self-adjoint endomorphism \(\operatorname{Hess}f(x):T_x\Man\to T_x\Man\).
The condition is then equivalent to
\(\operatorname{Hess}f(x)\succeq\mu\Id_{T_x\Man}\) on \(A\).
\end{definition}

\begin{definition}[Strong horospherical convexity
  {\citep{CriscitielloKim2025}}]
\label{def:strong-h-convexity}
Let \(A\subseteq\Man\) and let \(\mu>0\).  A function \(f:A\to\R\) is
\(\mu\)-strongly \(h\)-convex on \(A\) if, for every \(y\in A\), there exists
\(v\in T_y\Man\) such that, for every \(x\in A\),
\begin{equation}
\label{eq:h-minorant}
 f(x)-f(y)
 \ge
 Q_{y,v}^\mu(x)
 \defeq
 -\frac{\norm{v}_y^2}{2\mu}
 +\frac{\mu}{2}
 d\bigl(\Exp_y(-v/\mu),x\bigr)^2 .
\end{equation}
Equality holds at \(x=y\), since
\(d\bigl(y,\Exp_y(-v/\mu)\bigr)=\norm{v}_y/\mu\).  The point
\(\Exp_y(-v/\mu)\) is allowed to lie outside \(A\).
\end{definition}

Strong \(h\)-convexity replaces the tangent-space quadratic lower model in
\eqref{eq:strong-g-convexity} by the intrinsic squared-distance model
\(f(y)+Q_{y,v}^\mu\).  The vector \(v\) is an \(h\)-subgradient at \(y\).
For a differentiable globally \(h\)-convex function it is necessarily
\(\grad f(y)\) \citep[Proposition~2(v)]{CriscitielloKim2025}.  In our
ball-restricted differentiable class, we impose this canonical choice also at
boundary points.  In Euclidean space the model coincides with the standard
quadratic lower model.  On a Hadamard manifold, it is \(\mu\)-strongly \(g\)-convex, so
strong \(h\)-convexity implies strong \(g\)-convexity with the same
strong-convexity parameter \(\mu\)
\citep[Proposition~2(iv)]{CriscitielloKim2025}.
The underlying non-strong notions are not equivalent on hyperbolic space
\citep[Appendix~B.3]{CriscitielloKim2025}.  The term
\emph{horospherical} comes from the limit \(\mu\rightarrow 0\), in which the
center \(\Exp_y(-v/\mu)\) recedes to infinity and \(Q_{y,v}^\mu\) converges
to a Busemann support function \citep[Proposition~1]{CriscitielloKim2025}.



\subsection{The two optimization classes}
\label{sec:problem-classes}

Fix an order \(p\ge2\), an initial point \(x_0\), and a radius budget
\(R>0\).  For a radius \(r>0\), write
\(Q_p(r)\defeq L_pr^{p-1}/\mu\). The common dimensionless \(p\)-th-order
condition parameter is
\begin{equation}
\label{eq:Qp}
  \Qp\defeq Q_p(R)=\frac{L_pR^{p-1}}{\mu}.
\end{equation}
The power \(R^{p-1}\) corresponds to the convention that the \(p\)-th
covariant derivative is \(L_p\)-Lipschitz.

\paragraph{The \(h\)-convex class.}
For an inflation factor \(\beta\ge1\), let
\(B_{\mathrm h}^{\beta}\defeq\overline B(x^\star,\beta R)\).  The class
\(\Fh(\Man;\mu,L_p,R)\) consists of objectives \(f\), defined on a
neighborhood of \(B_{\mathrm h}^{\beta}\), with a minimizer \(x^\star\)
satisfying \(d(x_0,x^\star)\le R\), such that
\(f\) is \(\mu\)-strongly \(h\)-convex on \(B_{\mathrm h}^{\beta}\) with
\(v=\grad f(y)\), and \(\nabla^p f\) is \(L_p\)-Lipschitz on
\(B_{\mathrm h}^{\beta}\). 
All \(h\)-convex oracle queries are made in \(B_{\mathrm h}^{\beta}\).
The corresponding upper theorem proves that its iterates remain inside it.
The \(p\)-th-order upper bound uses a fixed constant-factor inflation:
\(\beta\) may depend on \(p\), but not on \(\Qp\), the accuracy, or the
dimension.  Both \(h\)-convex upper bounds hold on general Hadamard
manifolds, without a lower sectional-curvature bound.

\paragraph{The \(g\)-convex class.}
Here the center of the working region is the given initial point.  Set
\begin{equation}
\label{eq:g-domains}
 X\defeq\overline B(x_0,2R),
 \qquad
 U\defeq\overline B(x_0,10R),
 \qquad
 D\defeq\operatorname{diam}(X)\le4R.
\end{equation}
The class
\(\Fg(\Man;\mu,L_p,R)\) consists of objectives \(f\), defined on a
neighborhood of \(U\), whose constrained minimizer
\(x^\star=\argmin_{x\in X}f(x)\) satisfies \(d(x_0,x^\star)\le R\), and
such that \(f\) is \(\mu\)-strongly \(g\)-convex on \(X\), and
\(\nabla^p f\) is \(L_p\)-Lipschitz on \(U\).
Algorithms minimize over \(X\), and every oracle query must lie in \(U\).

For the \(g\)-convex upper bound we additionally assume that, for some
\(\kappa\ge0\), \(-\kappa^2\le\Sec_{\Man}\le0\), and quantify its
geometric distortion, following \citet[Lemma~5]{ZhangSra2016}, for
\(s\ge0\), by
\begin{equation}
\label{eq:curvature-factors}
 \zeta_\kappa(s)
 \defeq
 \begin{cases}
   \kappa s\coth(\kappa s),&\kappa s>0,\\
   1,&\kappa s=0,
 \end{cases}
 \qquad
 \xi_D\defeq4\zeta_\kappa(2D)-3.
\end{equation}
We consider two curvature regimes.
\begin{itemize}
  \item \textit{Bounded curvature.} Fix a constant \(\Xi\ge1\) and assume
  \(\xi_D\le\Xi\), with \(\Xi\) independent of \(\Qp\) and the target
  accuracy. Since \(D\le4R\), the simpler condition \(\kappa R=O(1)\) is
  sufficient for this regime.
  \item \textit{Growing curvature.} The lower bound concerns the hyperbolic
  space \(\Hyp^n_\kappa\) of constant curvature \(-\kappa^2\), where
  \(\kappa R\to\infty\).
\end{itemize}

The two working-region conventions reflect the algorithms. The ball
\(B_{\mathrm h}^{\beta}\) is an analysis region for the \(h\)-convex method:
its center \(x^\star\) is unknown, and the upper-bound proof shows that every
query lies in this ball. In contrast, \(X\) and \(U\) are public regions for
the \(g\)-convex method. Thus both models impose \(p\)-th-order regularity
only on a region containing all oracle queries.

\subsection{Exact information oracles}
\label{sec:oracle-model}

The public problem data include \(x_0,\mu,L_p,R\), the fixed inflation
factors, and any curvature bound appearing in the relevant theorem.  The
instance-dependent minimizer \(x^\star\) is not supplied to the method.
At a query point \(x\), the exact Riemannian \(p\)-th-order oracle returns
\begin{equation}
\label{eq:riemannian-oracle}
 \mathcal O_x^{(p)}f
 \defeq
 \bigl(f(x),\grad f(x),\nabla^2f(x),\ldots,\nabla^pf(x)\bigr).
\end{equation}
A deterministic method chooses each query as a function of the history of
preceding queries and oracle replies and, after \(T\) calls, outputs a point
determined by this history.
One lower-order query is considered as one \(p\)-th-order query.

Following the Euclidean higher-order lower-bound literature
\citep{AgarwalHazan2018,ArjevaniShamirShiff2019,GargKothariNetrapalliSherif2021},
we measure \(p\)-th-order oracle complexity.
Standard geometric primitives and arbitrary finite-dimensional computation
based on the returned derivative information incur no additional oracle
cost. For the \(h\)-convex upper bounds, these uncharged geometric operations
also include evaluating the Busemann functions used in
\cref{alg:h-pth-epoch} and solving its Busemann subproblem. The results
therefore bound oracle complexity, not polynomial arithmetic complexity or
bit complexity.

\subsection{Accuracy and complexity conventions}
\label{sec:complexity-conventions}

Write \(f^\star=f(x^\star)\).  An \textbf{\(\boldsymbol{\eps}\)-solution}
satisfies \(f(\widehat x)-f^\star\le\eps\), and its deterministic oracle
complexity is the worst-case number of oracle
calls over the relevant class.  In both classes, \(x^\star\) lies in the
interior of the relevant ball, so \(\grad f(x^\star)=0\).  Substituting
\(x^\star\) into either \eqref{eq:strong-g-convexity} or
\eqref{eq:h-minorant} gives
\begin{equation}
\label{eq:gap-distance}
  f(x)-f^\star\ge\frac{\mu}{2}d(x,x^\star)^2.
\end{equation}
\paragraph{Notations.}
Unless stated otherwise, \(p\) is fixed, lower bounds concern deterministic
algorithms in sufficiently high dimension, and constants hidden by
\(O_{p,\Xi}(\cdot)\) or \(\Omega_{p,\Xi}(\cdot)\) may depend on \(p\) and on
the bounded-curvature constant \(\Xi\), but not on \(\Qp\), \(\eps\), or the dimension.
Statements involving growing \(\kappa R\) display that dependence
separately.  \(\widetilde O\) or \(\widetilde \Omega\) suppresses only logarithmic
factors in the displayed dimensionless parameters. We write
\(\log_+t\defeq\max\{0,\log t\}\).

\section{Main results}
\label{sec:results}

The results separate three geometric regimes. In Euclidean space, strong
\(h\)-convexity and strong \(g\)-convexity both reduce to ordinary strong
convexity. Under bounded curvature, both classes retain the Euclidean-optimal
\(\Qp^{2/(3p+1)}\) rate. When \(\kappa R\) grows on hyperbolic space, the
two theories diverge. Horospherical supports localize the minimizer to the
curvature scale, whereas the larger \(g\)-convex class contains instances
with a curvature-dependent lower bound.
An upper bound using the additional \(h\)-convex structure may improve with
curvature, while a lower bound for the \(g\)-convex class may be attained by
instances without that structure. \Cref{tab:phase-diagram}
summarizes the resulting comparison.


\begin{table}[t]
\centering
\caption{Curvature-dependent higher-order complexity. In Euclidean space,
strong \(h\)- and \(g\)-convexity reduce to the strongly convex class in the
first row \citep{KornowskiShamir2020}.
The nonconvex results on manifolds use the pullback \(p\)-th-order model \citep{GutmanLobo2026}, 
where \(L_1\) and \(L_2\) are its first- and second-order pullback regularity
constants.
The convex results in this work use the exact Riemannian \(p\)-th-order oracle. 
The growing-curvature row uses the family with
\(Q_p\asymp_p(1+\kappa R)^p\). The table displays condition and curvature
dependence, while the formal upper theorems also state the dependence on the
target accuracy.}
\label{tab:phase-diagram}
\footnotesize
\renewcommand{\arraystretch}{1.3}
\begin{tabular}{@{}>{\raggedright\arraybackslash}p{0.95in}
                    >{\raggedright\arraybackslash}p{0.72in}
                    >{\raggedright\arraybackslash}p{0.67in}
                    >{\raggedright\arraybackslash}p{2.00in}
                    >{\raggedright\arraybackslash}p{1.40in}@{}}
\toprule
Class & Geometry & Curvature & Guarantee & Reference \\
\midrule
strongly convex & \(\R^m\) & flat &
\(\widetilde\Theta_p(\Qp^{2/(3p+1)})\) &
\citep{KornowskiShamir2020} \\
nonconvex & Riemannian & general &
\makecell[l]{\(O((L_1/\eps_1)^{(p+1)/p})^{\dagger}\)\\
\(O(\max\{(L_1/\eps_1)^{(p+1)/p},\)\\
\((L_2/\eps_2)^{(p+1)/(p-1)}\})^{\ddagger}\)} &
\citep{GutmanLobo2026} \\
\midrule
\multicolumn{5}{@{}l}{\textbf{\textit{This work}}} \\
strongly \(h\)-convex & Hadamard & \(\le0\) &
\(\widetilde O_p(1+\Qp^{2/(3p+1)})\) & \cref{thm:h-pth-oracle-upper} \\
strongly \(h\)-convex & \(\Hyp^n_\kappa\) & \(\kappa R\to\infty\) &
\(\widetilde O_p\!\bigl(
1+[\Qp/(\kappa R)^{p-1}]^{2/(3p+1)}
\bigr)\) &
\cref{thm:h-curvature-adaptive-upper} \\
strongly \(h\)-convex & \(\Hyp^n\) & \(-1\) &
\(\widetilde\Omega_p(\Qp^{2/(3p+1)})\) & \cref{thm:h-pth-oracle-lower} \\
strongly \(g\)-convex & Hadamard & \(\kappa R=O(1)\) &
\(\widetilde\Theta_{p,\Xi}(\Qp^{2/(3p+1)})\) &
\cref{thm:g-upper,thm:g-bounded-curvature-lower} \\
strongly \(g\)-convex & \(\Hyp^2_\kappa\) & \(\kappa R\to\infty\) &
\(\widetilde\Omega_p(\Qp^{1/p})\) &
\cref{thm:g-growing-curvature-lower} \\
\bottomrule
\end{tabular}
\par\vspace{2pt}
\begin{minipage}{\linewidth}
\footnotesize
\(\dagger\) \(\eps_1\)-stationarity. \quad
\(\ddagger\) \((\eps_1,\eps_2)\)-second-order stationarity for \(p\ge3\).
\end{minipage}
\renewcommand{\arraystretch}{1}
\end{table}

\subsection{Horospherically convex optimization}

The \(h\)-convex results consist of a general-Hadamard upper bound, its
curvature-adaptive refinement on hyperbolic space, and a fixed-curvature lower
bound. We begin with the general guarantee.

\begin{theorem}[Riemannian \(p\)-th-order \(h\)-convex upper bound]
\label{thm:h-pth-oracle-upper}
Fix \(p\ge2\).  There is a constant \(\beta_p\ge1\), depending only on
\(p\), and a deterministic exact Riemannian \(p\)-th-order method on
every Hadamard manifold \(\Man\) such that, for every
\(f\in\Fh(\Man;\mu,L_p,R)\) with \(\beta=\beta_p\), it returns an
\(\eps\)-solution after
\[
 \widetilde O_p\!\left(1+\Qp^{2/(3p+1)}\right)
 +O_p\!\left(\log\log\!\left(e^e+{\mu R^2}/{\eps}\right)\right)
\]
queries.
\end{theorem}

The second term matches the doubly logarithmic dependence on the target
accuracy in the Euclidean upper bound of \citet{KornowskiShamir2020}.

On hyperbolic space, the horoball procedure of
\citet[Proposition~8]{CriscitielloKim2025} localizes the minimizer to the
curvature scale. Combining it with \cref{thm:h-pth-oracle-upper} improves the
condition parameter.

\begin{theorem}[Curvature-adaptive \(h\)-convex upper bound]
\label{thm:h-curvature-adaptive-upper}
Fix \(p\ge2\), \(n\ge2\), and \(\kappa>0\). Suppose that
\(\Man=\Hyp^n_\kappa\) has sectional curvature \(-\kappa^2\). Define
\(\overline R_\kappa\defeq\min\{R,4/\kappa\}\) and
\(\overline Q_{p,\kappa}\defeq
L_p\overline R_\kappa^{p-1}/\mu\).
There is a deterministic exact Riemannian \(p\)-th-order method such that,
for every \(f\in\Fh(\Hyp^n_\kappa;\mu,L_p,R)\) with
\(\beta=\beta_p\), it returns an \(\eps\)-solution after
\[\widetilde O_p\!\left(1+
   \overline Q_{p,\kappa}^{2/(3p+1)}\right)
 +O_p\!\left(
   \log\log(e^e+\mu\overline R_\kappa^2/\eps)\right)
 +O(\log(1+\kappa R))
 \]
    queries.
In addition, every query lies in the original enlarged working ball.
\end{theorem}

The quantity \(\overline Q_{p,\kappa}\) is \(Q_p\) evaluated at the
localized radius \(\overline R_\kappa\), rather than a different condition
measure. Equivalently, \(\overline Q_{p,\kappa}
=\Qp\min\{1,4/(\kappa R)\}^{p-1}\). Thus
\(\overline Q_{p,\kappa}=\Qp\) when \(\kappa R\le4\), whereas
\(\overline Q_{p,\kappa}=4^{p-1}\Qp/(\kappa R)^{p-1}\) when
\(\kappa R\ge4\).
Thus the general-Hadamard guarantee is worst-case curvature-free, while
hyperbolic curvature can make the \(h\)-convex problem easier. More
precisely, after \(O(\log(1+\kappa R))\) localization queries,
\cref{thm:h-curvature-adaptive-upper} replaces the leading term
\(\widetilde O_p(1+\Qp^{2/(3p+1)})\) in
\cref{thm:h-pth-oracle-upper} by
\(\widetilde O_p(1+\overline Q_{p,\kappa}^{2/(3p+1)})\). When
\(\kappa R\ge4\), this is
\(\widetilde O_p(1+[\Qp/(\kappa R)^{p-1}]^{2/(3p+1)})\), up to constants
depending only on \(p\).

The next theorem gives the matching fixed-curvature lower bound on
hyperbolic space.

\begin{theorem}[Hyperbolic \(p\)-th-order \(h\)-convex lower bound]
\label{thm:h-pth-oracle-lower}
Fix \(p\ge2\). Let
\((\mu_{\rm E},L_{p,\rm E},R_{\rm E},\eps)\) satisfy the
Euclidean-transfer regime of \cref{app:euclidean-transfer} and put
\(Q_{p,\rm E}=L_{p,\rm E}R_{\rm E}^{p-1}/\mu_{\rm E}\). For
\(Q_{p,\rm E}\ge q_p\), there is an instance
\(f\in\Fh(\Hyp^n;\mu,L_p,R)\), where \(\Hyp^n\) has sectional
curvature \(-1\), with resulting class parameter
\(\Qp=\Theta_p(1+Q_{p,\rm E})\), such that every deterministic exact
Riemannian \(p\)-th-order method requires
\(\Omega_p(\Qp^{2/(3p+1)})\) queries to return an \(\eps\)-solution. The
dimension satisfies \(n\ge c_p\Qp^{2/(3p+1)}\).
\end{theorem}

Taken together,
\cref{thm:h-pth-oracle-upper,thm:h-pth-oracle-lower} give the
Euclidean-optimal \(\Qp\) exponent at fixed curvature, up to logarithmic
factors, for the same exact Riemannian \(p\)-th-order oracle and the same
constant-factor working-ball inflation. The lower bound is a fixed-curvature
statement. It does not oppose the curvature-adaptive upper bound, which
reduces the effective condition parameter when \(\kappa R\) grows.

\subsection{Geodesically convex optimization}

The \(g\)-convex theory has two regimes. Under bounded curvature, matching
upper and lower bounds recover the Euclidean-optimal \(\Qp\) exponent. 
When $\kappa R$ grows on pure hyperbolic space, a separate lower
bound rules out the same uniform exponent over the full \(g\)-convex class.

\begin{theorem}[Bounded-curvature \(g\)-convex upper bound]
\label{thm:g-upper}
Fix \(p\ge2\).  Suppose that \(-\kappa^2\le\Sec_{\Man}\le0\) and
\(\xi_D\le\Xi\) for a fixed constant \(\Xi\ge1\).  There is a
deterministic exact Riemannian \(p\)-th-order method such that, for every
\(f\in\Fg(\Man;\mu,L_p,R)\), it returns an \(\eps\)-solution after
\[
 \widetilde O_p\!\left(
 \xi_D^{2(p+1)/(3p+1)}\Qp^{2/(3p+1)}
 +\xi_D
 +\log\log(e^e+ {\mu R^2}/{\eps})
 \right)
\]
queries.  In particular, for fixed \(\Xi\), the \(\Qp\)-dependence is
\(\widetilde O_{p,\Xi}(\Qp^{2/(3p+1)})\).
\end{theorem}

\begin{theorem}[Bounded-curvature \(g\)-convex lower bound]
\label{thm:g-bounded-curvature-lower}
Fix \(p\ge2\) and \(\Xi>1\). Let
\((\mu_{\rm E},L_{p,\rm E},R_{\rm E},\eps)\) satisfy the
Euclidean-transfer regime of \cref{app:euclidean-transfer} and put
\(Q_{p,\rm E}=L_{p,\rm E}R_{\rm E}^{p-1}/\mu_{\rm E}\). For
\(Q_{p,\rm E}\ge q_{p,\Xi}\), there are Hadamard manifolds of the
form \(\R^m\times\Hyp^2_\kappa\), with \(\xi_D\le\Xi\), and objectives in
the corresponding \(g\)-convex class whose resulting parameter satisfies
\(\Qp=\Theta_{p,\Xi}(1+Q_{p,\rm E})\). Every deterministic exact Riemannian
\(p\)-th-order method requires
\(\Omega_{p,\Xi}(\Qp^{2/(3p+1)})\) queries to return an \(\eps\)-solution,
provided \(m\ge c_p\Qp^{2/(3p+1)}\). For \(\Xi=1\), the same conclusion
holds in Euclidean space.
\end{theorem}

\cref{thm:g-upper,thm:g-bounded-curvature-lower} establish optimality only in the
\(\Qp\) exponent for fixed \(\Xi\).  They do not claim that the displayed
dependence on \(\xi_D\) in \cref{thm:g-upper} is optimal.

\begin{theorem}[Growing-curvature lower bound]
\label{thm:g-growing-curvature-lower}
Fix \(p\ge2\).  For every sufficiently large \(\kappa R\), there is a
\(\mu\)-strongly \(g\)-convex exact Riemannian \(p\)-th-order hard family in
\(\Fg(\Hyp^2_\kappa;\mu,L_p,R)\) with the valid class budget
\(L_p=C_p\mu\kappa^{p-1}(1+\kappa R)\) and
\(\Qp\asymp_p(1+\kappa R)^p\),
such that every deterministic method making at most
\(c_p\kappa R/\log(2+\kappa R)\) queries has an instance on which its output is at distance at least
\(R/4\) from the minimizer.  Hence, for every
\(\eps<\mu R^2/32\), the query complexity is
\[\Omega_p(\kappa R/\log(2+\kappa R))
=\widetilde\Omega_p(\Qp^{1/p})\]
\end{theorem}

The smoothness constant in \cref{thm:g-growing-curvature-lower} is a valid
upper bound and need not be minimal. The \(\Qp^{1/p}\) form rewrites the
bound for the family with \(\Qp\asymp_p(1+\kappa R)^p\). 
Since \(1/p>2/(3p+1)\), the result rules out a curvature-independent extension
of the bounded-curvature rate. It does not give a matching growing-curvature
upper bound. The hard family is not contained in the strongly \(h\)-convex
class with the corresponding constant-factor working region. Otherwise,
\cref{thm:h-pth-oracle-upper} would give
\(\widetilde O_p((\kappa R)^{2p/(3p+1)})\) queries along the balanced
scaling above, contradicting the
\(\widetilde\Omega_p(\kappa R)\) lower bound because \(2p/(3p+1)<1\).
The unperturbed squared-distance objectives are strongly \(h\)-convex. Thus
the separation is created by the resisting perturbations.

\section{Horospherically convex optimization}
\label{sec:h-convex}

The upper-bound proofs use Busemann functions as intrinsic affine
minorants. Let \(\gamma:[0,\infty)\to\Man\) be a geodesic ray parametrized by
arclength, meaning that
\(\norm{\gamma'(t)}_{\gamma(t)}=1\) for every \(t\ge0\). Its associated
Busemann function is
\(b_\gamma(z)\defeq\lim_{t\to\infty}(d(z,\gamma(t))-t)\).
This function is geodesically convex
\citep[Section~2.3]{CriscitielloKim2025}. If \(\gamma(0)=y\), the
squared-distance inequality applied to \(\gamma(T)\) and then passed to the
limit gives
\begin{equation}
\label{eq:busemann-basepoint}
 b_\gamma(z)\ge
 -\inner{\gamma'(0)}{\Log_yz}_y
 \qquad (z\in\Man).
\end{equation}

\begin{lemma}[Busemann minorant]
\label{lem:h-busemann-minorant}
Let \(f\) be \(\mu\)-strongly \(h\)-convex on
\(B_{\mathrm h}^{\beta}\) and let \(y\in B_{\mathrm h}^{\beta}\). Put
\(v=\grad f(y)\). If \(v\ne0\), define
\(\gamma_y(t)\defeq\Exp_y(-tv/\norm{v}_y)\) and \(b_y\defeq b_{\gamma_y}\).
Then \(m_y(z)\defeq f(y)+\norm{v}\,b_y(z)\le f(z)\) for every
\(z\in B_{\mathrm h}^{\beta}\).
\end{lemma}

\begin{proof}
Set \(L=\norm{v}/\mu\). Strong \(h\)-convexity gives
\(f(z)-f(y)\ge \frac{\mu}{2}(d(\gamma_y(L),z)^2-L^2)\).
The function \(t\mapsto d(\gamma_y(t),z)-t\) is nonincreasing. Hence
\(d(\gamma_y(L),z)-L\ge b_y(z)\). If \(L+b_y(z)\ge0\), then
\[
 \frac{\mu}{2}\bigl(d(\gamma_y(L),z)^2-L^2\bigr)
 \ge \mu Lb_y(z)+\frac{\mu}{2}b_y(z)^2
 \ge \norm{v}\,b_y(z).
\]
If \(L+b_y(z)<0\), the right-hand side is smaller than
\(-\mu L^2\), while the left-hand side is at least
\(-\mu L^2/2\). The same conclusion follows.
\end{proof}

This finite-\(\mu\) comparison agrees with the Busemann limit of the strong
\(h\)-convex lower model
\citep[Proposition~1]{CriscitielloKim2025}.

\subsection{\texorpdfstring{\(p\)-th-order oracle upper bound}{p-th-order oracle upper bound}}
\label{sec:h-pth-oracle-upper}

This subsection works on an arbitrary Hadamard manifold \(\Man\). The
Euclidean higher-order oracle rate is due to \citet{KornowskiShamir2020}. We
adapt the corresponding proof framework to Hadamard manifolds by replacing
Euclidean affine minorants with Busemann minorants in an intrinsic estimate
sequence. The matching lower bound remains specific to hyperbolic space.

The method couples a Busemann estimate sequence with points satisfying a
covariant Taylor-residual condition. For \(u\in\Man\), define
\({\cal G}_{p,u}(s)\defeq
\sum_{j=0}^{p-1}(\nabla^j\grad f)_u[s^{\otimes j}]/j!\) and let
\begin{equation}
\label{eq:h-residual-selector}
 {\cal S}(u,\lambda)\in\argmin_{x\in\Man}
 \left\{d(u,x)\,\middle|\,
 \left\|{\cal G}_{p,u}(\Log_u x)
 -\PT_{x\to u}\!\left(\frac{\Log_xu}{\lambda}\right)\right\|_u
 \le\frac{L_p}{p!}d(u,x)^p\right\},
\end{equation}
using a fixed public deterministic tie-breaking rule.
Fix \(\delta_p=1/(8p)\). At \(k=0\), start from \(\lambda=1/\mu\), halve
until the statistic \(\lambda L_p\rho^{p-1}\) is at most \(\theta_+\), then
double to obtain a low--high bracket and bisect until
\(\lambda_L/\lambda_H\ge1-\delta_p\). At \(k\ge1\), restrict
\(\alpha_k\le\bar\alpha_k\defeq
1-(1+1/k)^{-(3p+1)/2}\). Accept the cap when its statistic is at most
\(\theta_+\). Otherwise halve \(\alpha\) to obtain a low--high bracket and
bisect until \(\lambda_L/\lambda_H\ge1-\delta_p\) and
\(2|\alpha_H-\alpha_L|d(x_k,\omega_k)\le\delta_p\rho_H\).
The proof of \cref{prop:h-pth-epoch} shows that the accepted low endpoint
satisfies \(\theta_-\le\lambda_kL_p\rho_k^{p-1}\le\theta_+\), for a fixed
\(\theta_->0\) depending only on \(p\). The full method restarts these
epochs while the localized condition parameter is above a fixed constant
depending only on \(p\). It then switches to the prescribed local step in
\cref{lem:h-local-step}. We first state the epoch procedure.

\begin{algorithm}[H]
\caption{Riemannian \(p\)-th-order epoch for \(h\)-convex optimization}
\label{alg:h-pth-epoch}
\begin{algorithmic}
  \Require Epoch center \(z\), accepted-step budget \(N\), and
  \(0<\theta_-<\theta_+\) with \(2\theta_+/p!\le1/(p+1)\)
  \State \(A_0\gets0\), \(x_0\gets\omega_0\gets z\)
  \For{\(k=0,\ldots,N-1\)}
    \State Choose \(a_{k+1}>0\) by the parameter search above
    \State \(A_{k+1}\gets A_k+a_{k+1}\),
    \(\alpha_k\gets a_{k+1}/A_{k+1}\),
    \(\lambda_k\gets a_{k+1}^2/A_{k+1}\)
    \State \(\widetilde y_k\gets\gamma_{x_k,\omega_k}(\alpha_k)\),
    \(\widetilde x_k\gets{\cal S}(\widetilde y_k,\lambda_k)\)
    \State \(\rho_k\gets d(\widetilde y_k,\widetilde x_k)\),
    \(w_k\gets\lambda_k^{-1}\Log_{\widetilde x_k}\widetilde y_k\)
    \State Query \(v_k\gets\grad f(\widetilde x_k)\)
    \State If \(v_k=0\), \textbf{return} \(\widetilde x_k\)
    \State \(\eta_k(t)\gets
    \Exp_{\widetilde x_k}(-tv_k/\norm{v_k})\),
    \(b_k\gets b_{\eta_k}\)
    \State \(\omega_{k+1}\gets\argmin_{x\in\Man}
    \{d(x,\omega_k)^2/2+a_{k+1}\norm{v_k}b_k(x)\}\)
    \State \(x_{k+1}\gets\widetilde x_k\)
  \EndFor
  \State \textbf{return} \(x_N\)
\end{algorithmic}
\end{algorithm}

The next lemma shows that \(\mathcal S\) is well-defined and satisfies the
relative-error condition used below.

\begin{lemma}[Relative-error condition]
\label{lem:h-residual-step}
Let \(u\in B_{\mathrm h}^{\beta}\), let \(\lambda>0\), and let
\(\widehat x\in\argmin_{x\in B_{\mathrm h}^{\beta}}
\{f(x)+d(u,x)^2/(2\lambda)\}\).
Suppose that \(\widehat x\) lies in the interior of
\(B_{\mathrm h}^{\beta}\) and
\(\overline B(u,d(u,\widehat x))\subseteq B_{\mathrm h}^{\beta}\).
Then the minimum in \eqref{eq:h-residual-selector} is attained. If
\(\widetilde x={\cal S}(u,\lambda)\), then, after querying
\(v=\grad f(\widetilde x)\) and putting
\(\rho\defeq d(u,\widetilde x)\) and
\(w\defeq\lambda^{-1}\Log_{\widetilde x}u\),
we have
\begin{equation}
\label{eq:h-relative-residual}
 \norm{v-w}
 \le\frac{2L_p}{p!}\rho^p
 =\sigma\norm{w},
 \qquad
 \sigma\defeq\frac{2\lambda L_p\rho^{p-1}}{p!}.
\end{equation}
\end{lemma}

\begin{proof}
The point \(\widehat x\) is used only as a witness. The strong
\(h\)-convexity inequality implies strong
\(g\)-convexity with the same parameter. Resolvent nonexpansiveness
\citep{Bacak2014} gives \(d(\widehat x,x^\star)\le d(u,x^\star)\).
Since \(\widehat x\) is interior,
\(\grad f(\widehat x)=\lambda^{-1}\Log_{\widehat x}u\).
The covariant Taylor remainder along the geodesic from \(u\) to
\(\widehat x\) gives
\[
 \norm{
 \PT_{\widehat x\to u}\grad f(\widehat x)
 -{\cal G}_{p,u}(\Log_u\widehat x)
 }_u
 \le\frac{L_p}{p!}d(u,\widehat x)^p.
\]
Thus \(\widehat x\) is feasible for \eqref{eq:h-residual-selector}. The feasible set is closed. Its
intersection with
\(\overline B(u,d(u,\widehat x))\) is compact by Hopf--Rinow, so a closest
feasible point exists and
\begin{equation}
\label{eq:h-witness-distance}
 d(u,\widetilde x)\le d(u,\widehat x).
\end{equation}
A second Taylor remainder, now at \(\widetilde x\), and the residual
inequality in \eqref{eq:h-residual-selector} yield
\[
 \norm{
 \PT_{\widetilde x\to u}
 \left(\grad f(\widetilde x)
 -\lambda^{-1}\Log_{\widetilde x}u\right)
 }_u
 \le\frac{2L_p}{p!}\rho^p.
\]
Parallel transport is an isometry, which proves
\eqref{eq:h-relative-residual}.
\end{proof}

\begin{lemma}[Local \(p\)-th-order step]
\label{lem:h-local-step}
Suppose that \(r>0\), \(d(z,x^\star)\le r\), \(L_p>0\), and
\(\overline B(x^\star,3r)\subseteq B_{\mathrm h}^{\beta}\). Set
\(\lambda=(L_pr^{p-1})^{-1}\) and \(z^+={\cal S}(z,\lambda)\). Then
\[
 f(z^+)-f^\star\le C_pL_pr^{p+1},
 \qquad
 d(z^+,x^\star)\le
 C_pr\sqrt{\frac{L_pr^{p-1}}{\mu}},
\]
where \(C_p\) depends only on \(p\).
\end{lemma}

\begin{proof}
Let \(\widehat x\) minimize
\(f(x)+d(z,x)^2/(2\lambda)\) over \(B_{\mathrm h}^{\beta}\).
Resolvent nonexpansiveness gives
\(d(\widehat x,x^\star)\le r\), so \(\widehat x\) is interior and
\(d(z,\widehat x)\le2r\). Moreover,
\(\overline B(z,d(z,\widehat x))\subseteq
\overline B(x^\star,3r)\). Thus \cref{lem:h-residual-step} applies, and
\[
 \rho\defeq d(z,z^+)\le2r,
 \qquad
 \left\|\grad f(z^+)-\lambda^{-1}\Log_{z^+}z\right\|
 \le C_pL_pr^p.
\]

Put \(H(x)=f(x)+d(z,x)^2/(2\lambda)\). This function is
\((\mu+1/\lambda)\)-strongly geodesically convex. Its gradient at \(z^+\)
is the vector in the preceding display. The strong-convexity inequality and
Cauchy--Schwarz give
\[
 H(z^+)-H(\widehat x)
 \le\frac{\lambda}{2}\norm{\grad H(z^+)}^2.
\]
Since \(H(\widehat x)\le H(x^\star)\) and the regularizer is nonnegative,
\[
 f(z^+)-f^\star
 \le\frac{r^2}{2\lambda}
 +\frac{\lambda}{2}\norm{\grad H(z^+)}^2
 \le C_pL_pr^{p+1}.
\]
The distance bound follows from \eqref{eq:gap-distance}.
\end{proof}

\paragraph{Estimate sequence.}

After the update of \(\omega_{k+1}\) in \cref{alg:h-pth-epoch}, define
\begin{equation}
\label{eq:h-potential-update}
 \Psi_{k+1}(x)
 \defeq\Psi_k^\star+\frac12d(x,\omega_k)^2+a_{k+1}m_k(x),\qquad
 \Psi_{k+1}^\star\defeq\min_x\Psi_{k+1}(x),
\end{equation}
where
\(m_k(x)\defeq f(\widetilde x_k)+\norm{v_k}b_k(x)\).
The function \(\Psi_{k+1}\) is coercive and \(1\)-strongly geodesically
convex. Its unique minimizer in \eqref{eq:h-potential-update} is
\(\omega_{k+1}\).

\begin{lemma}[Estimate-sequence one-step bound]
\label{lem:h-estimate-one-step}
Suppose that
\begin{equation}
\label{eq:h-estimate-invariant}
 \Psi_k^\star+\frac12d(x,\omega_k)^2
 \le\frac12d(x,z)^2+A_kf(x)
\end{equation}
on the working ball and
\(\Delta_k\defeq\Psi_k^\star-A_kf(x_k)\ge0\).
Then \eqref{eq:h-estimate-invariant} holds at \(k+1\), and
\begin{equation}
\label{eq:h-one-step-reserve}
 \Delta_{k+1}
 \ge
 \Delta_k+
 \frac{1-\sigma_k^2}{2}
 A_{k+1}\lambda_k\norm{w_k}^2.
\end{equation}
Both conclusions include the first step \(\alpha_0=1\).
\end{lemma}

\begin{proof}
By \cref{lem:h-busemann-minorant}, \(m_k\le f\) on the working ball.
Thus \eqref{eq:h-estimate-invariant} and
\eqref{eq:h-potential-update} give
\(\Psi_{k+1}(x)\le d(x,z)^2/2+A_{k+1}f(x)\).
The function \(\Psi_{k+1}\) is \(1\)-strongly geodesically convex.
Consequently,
\(\Psi_{k+1}^\star+d(x,\omega_{k+1})^2/2\le\Psi_{k+1}(x)\),
which proves the next instance of \eqref{eq:h-estimate-invariant}.

Since \(b_k\) is \(1\)-Lipschitz,
\[
 \min_x\left\{
 \frac12d(x,\omega_k)^2+
 a_{k+1}\norm{v_k}b_k(x)
 \right\}
 \ge
 a_{k+1}\norm{v_k}b_k(\omega_k)
 -\frac{a_{k+1}^2}{2}\norm{v_k}^2.
\]
Using \(m_k(x_k)\le f(x_k)\), then Busemann convexity along
\([x_k,\omega_k]\), gives
\begin{align}
\label{eq:h-reserve-intermediate}
 \Delta_{k+1}
 &\ge
 \Delta_k+
 A_{k+1}\norm{v_k}b_k(\widetilde y_k)
 -\frac{A_{k+1}\lambda_k}{2}\norm{v_k}^2.
\end{align}
The basepoint inequality \eqref{eq:busemann-basepoint}, applied to
\(\eta_k\), gives
\(b_k(\widetilde y_k)\ge
\lambda_k\inner{v_k/\norm{v_k}}{w_k}_{\widetilde x_k}\).
Finally,
\[
 \inner{v_k}{w_k}
 -\frac12\norm{v_k}^2
 =
 \frac12\left(
 \norm{w_k}^2-\norm{v_k-w_k}^2
 \right)
 \ge\frac{1-\sigma_k^2}{2}\norm{w_k}^2
\]
by \eqref{eq:h-relative-residual}. Substitution into
\eqref{eq:h-reserve-intermediate} proves
\eqref{eq:h-one-step-reserve}. The proof does not divide by \(A_k\), so it
also applies when \(A_0=0\).
\end{proof}

\paragraph{Epoch complexity.}

\begin{proposition}[Riemannian \(p\)-th-order epoch]
\label{prop:h-pth-epoch}
There are constants \(C_p,\beta_p\ge1\), depending only on \(p\), with the
following property. If \(d(z,x^\star)\le D\), the method
\cref{alg:h-pth-epoch} returns a point \(x^+\) satisfying
\begin{equation}
\label{eq:h-epoch-target}
 f(x^+)-f^\star\le\frac{\mu D^2}{8}
\end{equation}
after at most
\[
 C_p\left[
 1+\left(\frac{L_pD^{p-1}}{\mu}\right)^{2/(3p+1)}
 \right]
\]
accepted steps. The total number of \(p\)-th-order oracle calls is
\[
 \widetilde O_p\left(
 1+\left(\frac{L_pD^{p-1}}{\mu}\right)^{2/(3p+1)}
 \right).
\]
Every query lies in
\(\overline B(x^\star,\beta_pD)\).
\end{proposition}

\begin{proof}
Set \(\Psi_0^\star=0\) for the analysis. The initialization in
\cref{alg:h-pth-epoch} then satisfies \eqref{eq:h-estimate-invariant} and
\(\Delta_0=0\).
Evaluating the invariant at \(x^\star\), and using
\(\Psi_k^\star\ge A_kf(x_k)\), gives
\begin{equation}
\label{eq:h-value-center-bound}
 A_k\bigl(f(x_k)-f^\star\bigr)
 +\frac12d(\omega_k,x^\star)^2
 \le\frac{D^2}{2}.
\end{equation}
Telescoping \eqref{eq:h-one-step-reserve} and using
\(\sigma_k\le\bar\sigma_p\defeq2\theta_+/p!\) gives
\begin{equation}
\label{eq:h-movement-budget}
 \sum_{k<N}
 \frac{A_{k+1}}{\lambda_k}\rho_k^2
 \le\frac{D^2}{1-\bar\sigma_p^2}.
\end{equation}

We next prove the growth of \(A_k\). Put
\(s\defeq(3p+1)/2\) and \(r\defeq2/(p-1)\).
The first-step band and
\(\rho_0\le d(z,\widehat x_0)\le2D\), where \(\widehat x_0\) is the
regularized minimizer used in the proof of \cref{lem:h-residual-step}, give
\begin{equation}
\label{eq:h-first-weight}
 A_1=\lambda_0
 \ge\frac{\theta_-}{L_p(2D)^{p-1}}.
\end{equation}
For \(k\ge1\), let
\(\bar\alpha_k=1-(1+1/k)^{-s}\).
Let \({\cal B}\) be the steps at which the accepted value is below the cap,
and define
\[
 {\cal L}_k
 \defeq
 \sum_{\substack{j\in{\cal B}\\1\le j\le k}}
 \log\frac{1-\alpha_j}{1-\bar\alpha_j}.
\]
The cap product gives
\begin{equation}
\label{eq:h-cap-product}
 A_{k+1}=A_1(k+1)^s\exp(-{\cal L}_k).
\end{equation}
On \(k\in{\cal B}\), write
\(\vartheta_k=\lambda_kL_p\rho_k^{p-1}\ge\theta_-\) and
\[
 \chi_k
 \defeq\frac{A_{k+1}}{\lambda_k}\rho_k^2
 =
 \frac{\vartheta_k^r}{L_p^r}
 \alpha_k^{-(2+2r)}A_{k+1}^{-r}.
\]
For \(k\ge2s\), use \(\alpha_k\le s/k\),
\eqref{eq:h-cap-product}, and
\[
 2+2r=(p+1)r,\qquad (p+1-s)r=-1,
\]
to obtain
\begin{equation}
\label{eq:h-deficit-charge}
 \chi_k
 \ge
 \left(\frac{c_p}{A_1L_p}\right)^r
 \frac{\exp(r{\cal L}_k)}{k}.
\end{equation}
Moreover, for \(k\ge2s\),
\(\Delta{\cal L}_k
=\log((1-\alpha_k)/(1-\bar\alpha_k))
\le(\bar\alpha_k-\alpha_k)/(1-\bar\alpha_k)\le2s/k\).
The indices \(j<2s\) contribute only \(O_p(1)\) to \({\cal L}_k\), since
each increment is at most
\(\log(1/(1-\bar\alpha_j))=s\log(1+1/j)\le s\log2\). This finite
contribution is absorbed into the constant in \eqref{eq:h-deficit-total}.
At such a band step,
\[
 e^{r{\cal L}_k}-e^{r{\cal L}_{k-1}}
 \le r e^{r{\cal L}_k}\Delta{\cal L}_k
 \le C_p(A_1L_p)^r\chi_k
\]
by \eqref{eq:h-deficit-charge}. At a capped step the left-hand side is zero.
Summing and using \eqref{eq:h-movement-budget} yields
\begin{equation}
\label{eq:h-deficit-total}
 \exp(r{\cal L}_k)
 \le
 C_p\left[
 1+D^2(A_1L_p)^r
 \right].
\end{equation}
Write \(A_1=G/(L_pD^{p-1})\), absorbing \(p\)-dependent constants into
\(G\). By \eqref{eq:h-first-weight}, \(G\ge c_p\).
Since \((p-1)r=2\), \eqref{eq:h-deficit-total} gives
\({\cal L}_k\le\log\max\{G,1\}+C_p\).
Substitution into \eqref{eq:h-cap-product} proves
\begin{equation}
\label{eq:h-weight-growth}
 A_N\ge
 c_p\frac{N^{(3p+1)/2}}{L_pD^{p-1}}.
\end{equation}
This is the same scalar growth exponent as in Euclidean higher-order
acceleration \citep{BubeckJiangLeeLiSidford2019}.

It remains to justify localization and the searches. The cap satisfies
\begin{equation}
\label{eq:h-alpha-square-sum}
 \sum_{k\ge0}\alpha_k^2\le3p+3.
\end{equation}
Indeed, \(\bar\alpha_k\le\min\{1,s/k\}\). Splitting the series at
\(\lceil s\rceil\), and including \(\alpha_0=1\), gives
\eqref{eq:h-alpha-square-sum}. Let \(\widehat x_k\) be the regularized
minimizer at the current center. Strong convexity of the regularized objective
and resolvent nonexpansiveness give
\(d(\widetilde x_k,\widehat x_k)\le\sigma_k\rho_k\) and
\(d(\widehat x_k,x^\star)\le d(\widetilde y_k,x^\star)\).
Together with \eqref{eq:h-relative-residual} and Cauchy--Schwarz, this gives
\begin{align}
\label{eq:h-total-drift}
 \sum_{k<N}\sigma_k\rho_k
 &\le
 \bar\sigma_p
 \left(
 \sum_{k<N}\frac{A_{k+1}}{\lambda_k}\rho_k^2
 \right)^{1/2}
 \left(
 \sum_{k<N}\frac{\lambda_k}{A_{k+1}}
 \right)^{1/2} \notag\\
 &\le
 \frac{\bar\sigma_p\sqrt{3p+3}}
 {\sqrt{1-\bar\sigma_p^2}}D
 \defeq C_p^{\rm mov}D.
\end{align}
Moreover, geodesic convexity of the distance and the preceding two
resolvent bounds give
\[
 d(x_{k+1},x^\star)
 \le(1-\alpha_k)d(x_k,x^\star)
 +\alpha_kd(\omega_k,x^\star)+\sigma_k\rho_k.
\]
The prefix form of \eqref{eq:h-total-drift} gives
\(d(x_k,x^\star)\le D+C_p^{\rm mov}D\) and
\(d(\omega_k,x^\star)\le D\).
Set \(K_p=2+C_p^{\rm mov}\) and choose \(\beta_p>3K_p\).
Every trial center lies in
\(\overline B(x^\star,K_pD)\). The regularized minimizer
\(\widehat x_k\), which is used only in the proof, lies no farther from
\(x^\star\) than the center. By
the triangle inequality,
\(\overline B(\widetilde y_k,
d(\widetilde y_k,\widehat x_k))\subseteq
\overline B(x^\star,3K_pD)\subset B_{\mathrm h}^{\beta}\).
This verifies the hypotheses of \cref{lem:h-residual-step}. By
\eqref{eq:h-witness-distance}, the selected residual point also lies in
\(\overline B(x^\star,3K_pD)\). Its oracle query and Taylor remainder are
therefore licensed before \eqref{eq:h-total-drift} is updated.

At the first step, the oracle at \(z\) is reused for every trial
\(\lambda\). The statistic
\(\lambda L_pd(z,{\cal S}(z,\lambda))^{p-1}\) tends to zero as
\(\lambda\downarrow0\). If it stayed below \(\theta_+\) as
\(\lambda\to\infty\), feasibility in
\eqref{eq:h-residual-selector} would imply
\(\grad f(z)=0\), which is the terminal case. Hence doubling gives a finite
low-high bracket. If \(\lambda_L/\lambda_H\ge1-\delta_p\), the resolvent
identity and the observed-to-exact comparison give
\[
 \lambda_LL_pd(z,{\cal S}(z,\lambda_L))^{p-1}
 \ge
 \frac{(1-\delta_p)^p}{(1+2\theta_+/p!)^{p-1}}\theta_+.
\]
Thus bisection returns the first-step band. Only its accepted point is
queried for a gradient.

For \(k\ge1\), a trial \(\alpha\) uses
\(\lambda(\alpha)=A_k\alpha^2/(1-\alpha)\) and the center
\(\gamma_{x_k,\omega_k}(\alpha)\). If the cap is high, halve \(\alpha\)
until a low trial is found. Bisect until
\[
 \frac{\lambda(\alpha_L)}{\lambda(\alpha_H)}\ge1-\delta_p,
 \qquad
 2|\alpha_H-\alpha_L|d(x_k,\omega_k)
 \le\delta_p\rho_H,
 \qquad \delta_p=\frac1{8p}.
\]
Writing \(r_L,r_H\) for the two exact regularized displacements,
nonexpansiveness in the base point and the resolvent identity give
\(r_L\ge(1-2\delta_p)r_H\).
The low observed displacement is within the factor
\(1+2\theta_+/p!\) of \(r_L\), while \(\rho_H\le r_H\). Hence the accepted
low endpoint satisfies
\[
 \lambda(\alpha_L)L_p\rho_L^{p-1}
 \ge
 \frac{(1-\delta_p)(1-2\delta_p)^{p-1}}
 {(1+2\theta_+/p!)^{p-1}}\theta_+.
\]
This is the required fixed lower threshold.

For every step with \(A_k<4/\mu\),
\(\lambda_{\rm cap}\le C_p/\mu\). Localization gives
\(\rho\le2K_pD\), and a trial is low whenever
\(\lambda L_p(2K_pD)^{p-1}\le\theta_+\). Every high endpoint also satisfies
\(\rho_H\ge c_pD(L_pD^{p-1}/\mu)^{-1/(p-1)}\). These bounds show that
halving and bisection require
\[
 O_p\left(
 1+\log(2+L_pD^{p-1}/\mu)+\log(k+2)
 \right)
\]
trial queries. Therefore the total query count differs from the accepted
step count only by logarithmic factors.

Finally, \eqref{eq:h-value-center-bound} and
\eqref{eq:h-weight-growth} show that the target
\eqref{eq:h-epoch-target} holds once \(A_N\ge4/\mu\). If \(L_p=0\), use
the valid positive upper bound
\(\widehat L_p=\mu D^{1-p}\), which gives a constant epoch cost.
\end{proof}

\begin{proof}[Proof of \cref{thm:h-pth-oracle-upper}]
Assume first that \(L_p>0\). Apply \cref{prop:h-pth-epoch} with
\(D_j=2^{-j}R\), restarting from the previous epoch output, until
\(L_pD_j^{p-1}/\mu\) is below a sufficiently small constant depending only
on \(p\). By \eqref{eq:gap-distance}, each completed epoch returns a center
within distance \(D_{j+1}\) of \(x^\star\). The nonconstant epoch costs form
a geometric series, while the number of epochs and their search overheads
are polylogarithmic in \(2+\Qp\). Their total cost is therefore
\(\widetilde O_p(1+\Qp^{2/(3p+1)})\).

Let \(z_0\) be the first local center, let \(D_0\) be its distance bound,
and use the scale-dependent condition parameter \(Q_p(D_t)\). Increase
\(\beta_p\), if necessary, so
that \(\beta_p\ge3\). The first local ball is contained in the original
working ball, and the transition threshold below makes the radii decrease,
so this containment holds inductively. Set
\(z_{t+1}={\cal S}(z_t,(L_pD_t^{p-1})^{-1})\), apply
\cref{lem:h-local-step}, and use its distance conclusion as the next bound
\(D_{t+1}\). After enlarging \(C_p\) if necessary,
\[
 f(z_{t+1})-f^\star\le C_p\mu D_t^2Q_p(D_t),
 \qquad
 D_{t+1}\le C_pD_t\sqrt{Q_p(D_t)},
 \qquad
 Q_p(D_{t+1})\le C_pQ_p(D_t)^{(p+1)/2}.
\]
Choose the transition threshold so that
\(C_p^{2/(p-1)}Q_p(D_0)\le1/2\) and
\(C_p\sqrt{Q_p(D_0)}\le1/2\). Induction gives
\(C_p^{2/(p-1)}Q_p(D_t)\le
2^{-((p+1)/2)^t}\), and the radii decrease. Since \(D_t\le R\), the first
inequality above gives a doubly exponential objective gap. Every local query
remains in the original working ball. Each local step uses at most two oracle
calls and no parameter search, so
\(O_p(\log\log(e^e+\mu R^2/\eps))\) local steps suffice.

If \(L_p=0\), choose a positive valid upper bound \(\widehat L_p\) small
enough that the transition condition holds and
\(C_p\widehat L_pR^{p+1}\le\eps\). One application of
\cref{lem:h-local-step} then reaches the target accuracy. Combining the
global and local costs proves the theorem.
\end{proof}

\subsection{Curvature-adaptive localization on hyperbolic space}
\label{sec:h-curvature-adaptive}

We now prove \cref{thm:h-curvature-adaptive-upper}. The additional geometric
step uses only the gradient contained in the Riemannian \(p\)-th-order
oracle.
The next lemma is the curvature-rescaled form of
\citet[Proposition~8]{CriscitielloKim2025}.

\begin{lemma}[Curvature-scale localization]
\label{lem:h-curvature-localization}
Let \(f\in\Fh(\Hyp^n_\kappa;\mu,L_p,R)\), where \(n\ge2\), and let
\(d(z_0,x^\star)\le R\). Define
\(N_{\mathrm{loc}}\defeq
\lceil4\max\{0,\log(\kappa R/4)\}\rceil\).
There is a gradient-based procedure that uses at most
\(N_{\mathrm{loc}}\) queries and returns \(z_{\mathrm{loc}}\) satisfying
\(d(z_{\mathrm{loc}},x^\star)\le\min\{R,4/\kappa\}\).
Every localization query lies in \(\overline B(x^\star,R)\).
\end{lemma}

\begin{proof}
If \(\kappa R\le4\), take \(z_{\mathrm{loc}}=z_0\). Suppose now that
\(\kappa R>4\), and put \(r_k=Re^{-k/4}\). At \(z_k\), query
\(g_k=\grad f(z_k)\). If \(g_k=0\), strong geodesic convexity gives
\(z_k=x^\star\), so take \(z_{\mathrm{loc}}=z_k\). Otherwise define
\(z_{k+1}\defeq
\Exp_{z_k}(-(r_k/2)g_k/\norm{g_k})\).

Assume inductively that \(d(z_k,x^\star)\le r_k\) and
\(\kappa r_k\ge4\). Define
\(\gamma_k(t)\defeq\Exp_{z_k}(-tg_k/\norm{g_k})\) and
\(b_k\defeq b_{\gamma_k}\).
The Busemann support from \cref{lem:h-busemann-minorant} gives
\(b_k(x^\star)\le0\).
Thus \(x^\star\) lies in the supporting horoball at \(z_k\). Let
\(H_k^{\mathrm{ball}}\) be the horoball that contains
\(\overline B(z_k,r_k)\) and whose boundary is tangent to this ball at
\(\Exp_{z_k}(-r_kg_k/\norm{g_k})\). Then \(x^\star\) also lies in
\(H_k^{\mathrm{ball}}\).

The intersection of these two horoballs is contained in
\[
 \overline B\!\left(
 z_{k+1},
 \frac{1}{\kappa}\operatorname{arcosh}
 \left(e^{\kappa r_k/2}\right)\right).
\]
To see the constant, rescale to curvature \(-1\) and use an isometry of the
upper-half-plane model. Put \(z_k=(0,1)\), let the supporting horoball be
\(\{y\ge1\}\), and write \(a=e^{r_k/2}\). Then
\(z_{k+1}=(0,a)\) and the second horoball is
\(\{(x,y):x^2+y^2\le a^2y\}\). For a point in their intersection,
\[
 \cosh d\bigl((0,a),(x,y)\bigr)
 =\frac{x^2+y^2+a^2}{2ay}
 \le a.
\]
Rescaling the metric gives the displayed radius.

For \(t\ge4\), \(\operatorname{arcosh}(e^{t/2})
\le t/2+1\le e^{-1/4}t\).
It follows that \(d(z_{k+1},x^\star)\le e^{-1/4}r_k=r_{k+1}\).
For every executed step \(k<N_{\mathrm{loc}}\), one has
\(\kappa r_k>4\). At termination, \(r_{N_{\mathrm{loc}}}\le4/\kappa\).
The induction also gives
\(d(z_k,x^\star)\le r_k\le R\) at every queried point.
\end{proof}

\begin{proof}[Proof of \cref{thm:h-curvature-adaptive-upper}]
Apply \cref{lem:h-curvature-localization} and put
\(\overline R_\kappa=\min\{R,4/\kappa\}\). Starting from
\(z_{\mathrm{loc}}\), apply \cref{thm:h-pth-oracle-upper} with radius budget
\(\overline R_\kappa\). Its enlarged working ball is contained in the
original enlarged ball because both balls are centered at \(x^\star\).
Substituting the localized radius gives the claimed complexity.

For \(\kappa R\ge4\),
\(\overline Q_{p,\kappa}=4^{p-1}\Qp/(\kappa R)^{p-1}\).
For \(\kappa R\ge8\), one localized epoch returns \(x^+\) satisfying
\(f(x^+)-f^\star\le(\mu/8)(4/\kappa)^2
=2\mu/\kappa^2\le\mu R^2/32\).
\end{proof}

\paragraph{Fixed-curvature optimality.}
\label{sec:h-lower}

The reduction below uses the public initial point as the center of the
hyperbolic normal chart.  The minimizer therefore remains hidden.

\begin{proof}[Proof of \cref{thm:h-pth-oracle-lower}]
Let \(g\) be the Euclidean hard instance of
\citet[Section~3.2]{KornowskiShamir2020}.  Denote its public initial point by \(u_0\),
its minimizer by \(u^\star\), its initial-radius bound by \(\bar R\), and
its strong-convexity and smoothness parameters by \(\bar\mu\) and
\(\bar L_p\). Thus
\(\bar Q_p\defeq\bar L_p\bar R^{p-1}/\bar\mu\).
Let \(\beta=\beta_p\) be the inflation factor in
\cref{thm:h-pth-oracle-upper}.

We first record a uniform lower-derivative bound for the source family.
Its nonquadratic part is a rotated path sum with terms
\(|t|^{p+1}/(p+1)\), together with a linear term.  The path-incidence
operator has norm at most \(2\).  Generalized H\"older therefore gives,
for \(1\le j\le p\),
\begin{equation}
\label{eq:h-source-derivatives}
 \sup_{\overline B(u^\star,\beta\bar R)}\norm{D^jg}_{\mathrm{op}}
 \le C_{p,\beta}\left(
 \bar L_p\bar R^{p+1-j}
 +\boldsymbol 1_{\{j\le2\}}\bar\mu\bar R^{2-j}
 \right).
\end{equation}
For \(j\ge2\), this follows directly from the path sum.  The linear term
vanishes, and the quadratic term contributes only when \(j=2\).  The case
\(j=1\) follows by integrating the Hessian from \(u^\star\), where
\(\nabla g(u^\star)=0\).  This bound is uniform over the adaptive rotation
used in the source lower bound.

Choose a public \(s>0\) such that
\begin{equation}
\label{eq:h-public-scale}
 s\bar R\le\frac{c_{p,\beta}}{1+\bar Q_p},
\end{equation}
where \(c_{p,\beta}>0\) is chosen sufficiently small.  Define
\begin{equation}
\label{eq:h-public-transplant}
 \Phi_s(x)\defeq u_0+s^{-1}\Log_{x_0}(x),
 \qquad
 f_s\defeq g\circ\Phi_s,
 \qquad
 R_s\defeq s\bar R,
 \qquad
 \mu_s\defeq\frac{\bar\mu}{4s^2}.
\end{equation}
The minimizer of \(f_s\) is
\begin{equation}
\label{eq:h-hidden-minimizer}
 x_s^\star
 =\Exp_{x_0}\bigl(s(u^\star-u_0)\bigr),
 \qquad
 d(x_0,x_s^\star)\le R_s.
\end{equation}
Hence the initial point \(x_0\) is public while \(x_s^\star\) remains
hidden.

We next verify strong \(h\)-convexity on
\(\overline B(x_s^\star,\beta R_s)\).  Since \(\Exp_{x_0}\) is distance
expanding,
\begin{equation}
\label{eq:h-image-bound}
 \norm{\Phi_s(x)-u^\star}
 \le s^{-1}d(x,x_s^\star)
 \le\beta\bar R
 \qquad
 \bigl(x\in\overline B(x_s^\star,\beta R_s)\bigr).
\end{equation}
Put \(\rho=(\beta+1)s\bar R\).  The radial metric
\(dr^2+\sinh^2(r)d\theta^2\) in normal coordinates about \(x_0\)
\citep[Section~1.4.4]{Petersen2016} gives the
following dimension-free bounds when \(\rho\le1/2\).  For a unit vector
\(w\) at \(x\in\overline B(x_0,\rho)\), let
\(U=D\Log_{x_0}(x)[w]\). Then
\(\rho/\sinh\rho\le\norm U\le1\) and
\(\norm{\Gamma_x(U,U)}\le C\rho\).
Here \(\Gamma_x\) denotes the Christoffel map in these coordinates.
Combining these estimates with \eqref{eq:h-image-bound} yields
\begin{equation}
\label{eq:h-public-hessian}
 \operatorname{Hess}f_s
 \succeq\frac{\bar\mu}{2s^2}\Id=2\mu_s\Id,
 \qquad
 \norm{\grad f_s(y)}
 \le\frac{C_{p,\beta}\bar\mu(1+\bar Q_p)\bar R}{s}
 \quad
 \bigl(y\in\overline B(x_s^\star,\beta R_s)\bigr),
\end{equation}
after decreasing \(c_{p,\beta}\).  Indeed, for the vector \(U\) above, the
Hessian equals
\(s^{-2}D^2g[U,U]-s^{-1}\inner{\nabla g}{\Gamma_x(U,U)}\).
The first term is at least
\(\bar\mu/[s^2(\sinh\rho/\rho)^2]\), while the second has magnitude at
most
\(C_{p,\beta}\bar\mu(1+\bar Q_p)\bar R^2\).
Condition \eqref{eq:h-public-scale} absorbs the latter.

For \(a,b\in T_y\Hyp^n\) with
\(\norm a,\norm b\le2/3\), the normal-coordinate metric also gives
\begin{equation}
\label{eq:h-local-gram}
 d(\Exp_ya,\Exp_yb)^2
 \le\norm{a-b}^2+\frac12\norm{a\wedge b}^2.
\end{equation}
Here
\(\norm{a\wedge b}^2=\norm a^2\norm b^2-\inner{a}{b}^2\).
To see this, map the straight segment from \(a\) to \(b\) through
\(\Exp_y\). Its squared speed is
\(\norm{a-b}^2+
\bigl((\sinh r/r)^2-1\bigr)\norm{a\wedge b}^2/r^2\),
where \(r\le2/3\). The coefficient of the wedge term is less than
\(1/2\).

Fix \(y,z\in\overline B(x_s^\star,\beta R_s)\), and put
\(v=\grad f_s(y)\),
\(a=-v/\mu_s\), and \(b=\Log_yz\).  By
\eqref{eq:h-public-scale} and \eqref{eq:h-public-hessian},
\(\norm a,\norm b\le2/3\).  Strong geodesic convexity and
\eqref{eq:h-local-gram} now give
\begin{align*}
 f_s(z)-f_s(y)
 &\ge\inner{v}{b}+\mu_s\norm b^2 \\
 &\ge-\frac{\norm v^2}{2\mu_s}
 +\frac{\mu_s}{2}
 d\bigl(\Exp_y(-v/\mu_s),z\bigr)^2.
\end{align*}
Thus \(f_s\) is \(\mu_s\)-strongly \(h\)-convex on
\(\overline B(x_s^\star,\beta R_s)\).

It remains to transfer smoothness and oracle information.  In the public
normal coordinates,
\[
 f_s(\Exp_{x_0}X)=g(u_0+X/s),
 \qquad
 D^j(f_s\circ\Exp_{x_0})=s^{-j}D^jg.
\]
The metric and its coordinate derivatives through order \(p\) are bounded
on \(\norm X\le\rho\). At almost every point, expanding the covariant
derivatives gives
\[
 \norm{\nabla^{p+1}f_s}_{\mathrm{op}}
 \le C_p\sum_{j=1}^{p+1}s^{-j}
 \norm{D^jg}_{\mathrm{op}}
 \le\frac{C_{p,\beta}}{s^{p+1}}
 \left(\bar L_p+\frac{\bar\mu}{\bar R^{p-1}}\right),
\]
where \(\norm{D^{p+1}g}_{\mathrm{op}}\le\bar L_p\) is understood almost
everywhere. The second inequality uses \eqref{eq:h-source-derivatives} and
\(s\bar R\le1\). Integration along geodesics therefore gives the valid
budget
\begin{equation}
\label{eq:h-public-smoothness}
 L_{p,s}
 \defeq\frac{C_{p,\beta}}{s^{p+1}}
 \left(\bar L_p+\frac{\bar\mu}{\bar R^{p-1}}\right).
\end{equation}
Thus the argument does not assume an everywhere-defined \((p+1)\)-st
derivative of the source \(C^{p,1}\) function.

At a manifold query \(x\), one Euclidean query at \(\Phi_s(x)\) returns
all derivatives required by the covariant chain rule.  The remaining
coefficients depend only on the public map and the hyperbolic geometry.
The simulation is therefore exact and query for query.  Moreover,
\begin{equation}
\label{eq:h-public-parameter-gap}
 \frac{L_{p,s}R_s^{p-1}}{\mu_s}
 =4C_{p,\beta}(1+\bar Q_p),
 \qquad
 f_s(x)-f_s^\star
 =g(\Phi_s(x))-g^\star.
\end{equation}
The assumption \(\bar Q_p\ge q_p\) in \cref{thm:h-pth-oracle-lower}
permits a fixed threshold such that
\(\bar Q_p\ge q_p\) and
\(\bar Q_p\ge L_{p,s}R_s^{p-1}/(8C_{p,\beta}\mu_s)\).
The Euclidean lower bound therefore yields
\[
 T
 \ge c_p\bar Q_p^{2/(3p+1)}
 \ge c_p'
 \left(\frac{L_{p,s}R_s^{p-1}}{\mu_s}\right)^{2/(3p+1)}.
\]
The target accuracy satisfies \eqref{eq:euclidean-transfer-accuracy}, and
the source construction permits \(n\ge c_pT\). This completes the proof.
\end{proof}

\section{Bounded-curvature geodesically convex optimization}
\label{sec:g-upper-proof}

This section proves \cref{thm:g-upper}. The method is a higher-order
implementation of the projected accelerated scheme of
\citet{MartinezRubioPokutta2023}. Three ingredients are needed. A
Riemannian \(p\)-th-order query produces an approximate minimizer of a
regularized objective. A finite search couples its regularization parameter
to the observed displacement. The resulting steps satisfy an accelerated
potential inequality with distortion factor \(\xi=\xi_D\).

\subsection{\texorpdfstring{A regularized step from \(p\)-th-order
information}{A regularized step from p-th-order information}}
\label{sec:g-regularized-step}

For \(y\in X\), let
\(N_X(y)\defeq
\{n\in T_y\Man:\inner{n}{\Log_yz}_y\le0\text{ for every }z\in X\}\)
be the normal cone of \(X\). Given a base point \(x\in U\), a parameter
\(\lambda>0\), and a queried point \(w\in U\), define
\begin{equation}
\label{eq:g-taylor-field}
 {\cal G}_{p,w}(s)
 \defeq
 \sum_{j=0}^{p-1}\frac{1}{j!}
 (\nabla^j\grad f)_w[s^{\otimes j}]
 \qquad (s\in T_w\Man)
\end{equation}
and
\begin{equation}
\label{eq:g-taylor-residual}
 {\cal R}_{w;x,\lambda}(y)
 \defeq
 \inf_{n\in N_X(y)}
 \left\|
 {\cal G}_{p,w}(\Log_wy)
 +\PT_{y\to w}\!\left(-\frac{\Log_yx}{\lambda}+n\right)
 \right\|_w.
\end{equation}
The point used by the method is
\begin{equation}
\label{eq:g-selector}
 {\cal S}(w;x,\lambda)
 \in\argmin_{y\in X}
 \left\{d(w,y):
 {\cal R}_{w;x,\lambda}(y)
 \le\frac{L_p}{p!}d(w,y)^p\right\},
\end{equation}
with a fixed deterministic tie-breaking rule. This computation depends only
on the oracle reply at \(w\), the known geometry, and the explicit set \(X\).

\begin{lemma}[Finite-order regularized step]
\label{lem:g-selector}
Let \(I_X\) denote the indicator function of \(X\), and define
\(h_{x,\lambda}(y)\defeq
f(y)+d(x,y)^2/(2\lambda)+I_X(y)\) and
\(y^\star_{x,\lambda}\defeq\argmin h_{x,\lambda}\).
The feasible set in \eqref{eq:g-selector} is nonempty and its minimum is
attained. If \(w^+={\cal S}(w;x,\lambda)\), then
\begin{align}
 d(w,w^+)&\le d(w,y^\star_{x,\lambda}),
 \label{eq:g-selector-distance}\\
 \operatorname{dist}(0,\partial h_{x,\lambda}(w^+))
 &\le\frac{2L_p}{p!}d(w,w^+)^p,
 \label{eq:g-selector-subgradient}\\
 d(w^+,y^\star_{x,\lambda})
 &\le\frac{2\lambda L_p}{p!}
 d(w,y^\star_{x,\lambda})^p.
 \label{eq:g-selector-contraction}
\end{align}
Starting from \(w_0=x\) and setting
\(w_{t+1}={\cal S}(w_t;x,\lambda)\), a sufficiently small value of
\(\lambda L_pd(x,w_1)^{p-1}\) yields double-exponential convergence to
\(y^\star_{x,\lambda}\). After
\(O_p(\log\log(1/\Delta))\) further queries, the returned point
\(\widehat y\) satisfies
\begin{equation}
\label{eq:g-relative-regularized-error}
 h_{x,\lambda}(\widehat y)-h_{x,\lambda}(y^\star_{x,\lambda})
 \le
 \frac{\Delta}{78\lambda}d(x,y^\star_{x,\lambda})^2.
\end{equation}
\end{lemma}

\begin{proof}
The covariant Taylor formula along the geodesic from \(w\) to \(y\) gives
\begin{equation}
\label{eq:g-vector-taylor-remainder}
 \left\|
 \PT_{y\to w}\grad f(y)-{\cal G}_{p,w}(\Log_wy)
 \right\|_w
 \le\frac{L_p}{p!}d(w,y)^p.
\end{equation}
The optimality condition at \(y^\star_{x,\lambda}\), followed by
\eqref{eq:g-vector-taylor-remainder}, proves that this point is feasible in
\eqref{eq:g-selector}. Compactness of \(X\) and closedness of the normal-cone
graph give existence of a closest feasible point. This proves
\eqref{eq:g-selector-distance}. Applying
\eqref{eq:g-vector-taylor-remainder} again at \(w^+\) gives
\eqref{eq:g-selector-subgradient}. Since \(h_{x,\lambda}\) is
\(1/\lambda\)-strongly \(g\)-convex, the subgradient inequality gives
\eqref{eq:g-selector-contraction}.

Write \(r^\star=d(x,y^\star_{x,\lambda})\), \(s_1=d(x,w_1)\), and
\(\phi=\lambda L_ps_1^{p-1}\). The first two conclusions imply
\begin{equation}
 \label{eq:g-observed-exact-comparison}
 s_1\le r^\star\le(1+2\phi/p!)s_1.
\end{equation}
Writing \(e_t=d(w_t,y^\star_{x,\lambda})\), repeated use of
\eqref{eq:g-selector-contraction} gives
\(e_{t+1}\le(2\lambda L_p/p!)e_t^p\). If \(r^\star=0\), the feasible
witness \(x=y^\star_{x,\lambda}\) forces \(w_1=x\), and the conclusion is
immediate. Suppose \(r^\star>0\). By
\eqref{eq:g-observed-exact-comparison}, a sufficiently small upper bound on
\(\phi\) ensures
\[
 \frac{e_t}{r^\star}
 \le
 \left(\frac{2\lambda L_p(r^\star)^{p-1}}{p!}\right)^{-1/(p-1)}
 \left[
 \left(\frac{2\lambda L_p(r^\star)^{p-1}}{p!}\right)^{1/(p-1)}
 \frac{e_1}{r^\star}
 \right]^{p^{t-1}},
\]
where the quantity in brackets is strictly smaller than one. Thus the
computable residual bound in \eqref{eq:g-selector-subgradient} reaches
\[
 \lambda\operatorname{dist}(0,\partial h_{x,\lambda}(\widehat y))
 \le\sqrt{\Delta/156}\,d(x,\widehat y)
\]
after \(O_p(\log\log(1/\Delta))\) further queries. Requiring the left-hand
side also to be at most \(d(x,\widehat y)/2\) gives
\(d(\widehat y,y^\star_{x,\lambda})\le d(x,\widehat y)/2\), hence
\(d(x,\widehat y)\le2r^\star\). Strong convexity of
\(h_{x,\lambda}\) now gives
\[
 h_{x,\lambda}(\widehat y)-h_{x,\lambda}(y^\star_{x,\lambda})
 \le\frac{\lambda}{2}\operatorname{dist}
 (0,\partial h_{x,\lambda}(\widehat y))^2
 \le\frac{\Delta}{78\lambda}(r^\star)^2,
\]
which is \eqref{eq:g-relative-regularized-error}.
\end{proof}

\subsection{The bounded-curvature method}
\label{sec:g-method}

Fix \(3/4<\gamma<8/9\) and put
\(\bar q=(4\gamma-3)/6\). At an iteration with weight \(A>0\), projected
dual vector \(\bar u\), and current point \(y\), a trial \(q\in(0,\bar q]\)
sets
\begin{equation}
\label{eq:g-trial-parameters}
 a(q)\defeq qA,
 \qquad
 \lambda(q)\defeq\frac{Aq^2}{\gamma(\xi+q)},
 \qquad
 x(q)\defeq\Exp_y\!\left(\frac{q}{1+q}\bar u\right).
\end{equation}
Query at \(x(q)\), compute
\(w_1(q)={\cal S}(x(q);x(q),\lambda(q))\), and observe
\begin{equation}
\label{eq:g-search-statistic}
 \phi(q)\defeq
 \lambda(q)L_pd(x(q),w_1(q))^{p-1}.
\end{equation}
The search first tests \(\bar q\). If \(\phi(\bar q)\le\bar\vartheta_p\),
where \(\bar\vartheta_p>0\) depends only on \(p\), it accepts the cap.
Otherwise it halves \(q\) until it
finds a low trial and bisects the resulting interval. The accepted low
endpoint has
\begin{equation}
\label{eq:g-large-step-band}
 0<\theta_-\le
 \lambda_kL_pd(x_k,y_k)^{p-1}
 \le\theta_+,
\end{equation}
where \(\theta_-\) and \(\theta_+\) depend only on \(p\). The only exception
is an accepted cap with statistic below \(\theta_-\), which we call a
\emph{capped step}. In every trial, \cref{lem:g-selector} is iterated until
\eqref{eq:g-relative-regularized-error} holds with
\(\Delta_k=(k+1)^{-2}\).
The complete deterministic stopping rule and its query bound are given in
\cref{lem:g-parameter-search}.

The resulting epoch is given in \cref{alg:g-epoch}. Here
\(\Pi_{\overline B(0,D)}\) is Euclidean projection in the current tangent
space. The counter \(m\) records the number of steps satisfying
\eqref{eq:g-large-step-band}.

\begin{algorithm}[H]
\caption{Riemannian \(p\)-th-order epoch for bounded-curvature
\(g\)-convex optimization}
\label{alg:g-epoch}
\begin{algorithmic}
  \Require Center \(z\in X\), distance bound \(r\), large-step budget \(N\),
  and target weight \(A_{\rm tar}\)
  \State \(\lambda_0\gets c_p/[L_p(4r)^{p-1}]\),
  \(A_0\gets200\xi\lambda_0\)
  \State Fix \(\Delta_0\in(0,1)\). Starting at \(z\), iterate
  \eqref{eq:g-selector} for \(h_{z,\lambda_0}\) to the accuracy in
  \eqref{eq:g-relative-regularized-error} and call the result \(y_0\)
  \State \(u_0\gets0\), \(k\gets1\), \(m\gets0\)
  \While{\(m<N\) and \(A_{k-1}<A_{\rm tar}\)}
    \State \(\bar u_{k-1}\gets
    \Pi_{\overline B(0,D)}(u_{k-1})\)
    \State Use the finite search in \cref{lem:g-parameter-search} to obtain
    \(q_k,a_k,\lambda_k,x_k,y_k\)
    \State \(\widetilde u_k\gets
    \Log_{x_k}(\Exp_{y_{k-1}}\bar u_{k-1})\),
    \(v_k\gets-\lambda_k^{-1}\Log_{x_k}y_k\)
    \State \(u_k\gets
    \PT_{x_k\to y_k}(\widetilde u_k-a_kv_k/\xi)+\Log_{y_k}x_k\)
    \State \(A_k\gets A_{k-1}+a_k/\xi\)
    \State If \eqref{eq:g-large-step-band} holds, \(m\gets m+1\)
    \State \(k\gets k+1\)
  \EndWhile
  \State \textbf{return} \(y_{k-1}\)
\end{algorithmic}
\end{algorithm}

\begin{lemma}[Finite parameter search]
\label{lem:g-parameter-search}
The preceding search terminates after a logarithmic number of trial queries.
It returns either a step satisfying \eqref{eq:g-large-step-band} or a capped
step with
\begin{equation}
\label{eq:g-capped-growth}
 \frac{A_k}{A_{k-1}}=1+\frac{\bar q}{\xi}.
\end{equation}
Before \(A_k\ge A_{\rm tar}\), there are at most
\(O(\xi\log_+(A_{\rm tar}/A_0))\) capped steps.
\end{lemma}

\begin{proof}
Write
\(J_\lambda(u)\defeq
\argmin_{z\in X}\{f(z)+d(u,z)^2/(2\lambda)\}\) and
\(y^\star(q)=J_{\lambda(q)}(x(q))\). Set \(s(q)=d(x(q),w_1(q))\) and
\(r(q)=d(x(q),y^\star(q))\).
By \eqref{eq:g-observed-exact-comparison},
\begin{equation}
\label{eq:g-search-observed-exact}
 s(q)\le r(q),
 \qquad
 r(q)\le\left(1+\frac{2\phi(q)}{p!}\right)s(q)
 \quad\text{if }\phi(q)\le\bar\vartheta_p,
\end{equation}
where \(\bar\vartheta_p>0\) is a sufficiently small constant depending only
on \(p\).

We first record the comparison used by the bisection. The resolvent of the
convex function \(f+I_X\) is nonexpansive on a Hadamard manifold. Moreover,
if \(0<\lambda_L\le\lambda_H\), then
\begin{equation}
\label{eq:g-resolvent-parameter}
 d(J_{\lambda_L}(u),J_{\lambda_H}(u))
 \le\left(1-\frac{\lambda_L}{\lambda_H}\right)
 d(u,J_{\lambda_H}(u)).
\end{equation}
Indeed, for \(y=J_{\lambda_H}(u)\) and
\(z=\Exp_y((\lambda_L/\lambda_H)\Log_yu)\), the proximal optimality
condition gives \(y=J_{\lambda_L}(z)\). Nonexpansiveness then gives
\eqref{eq:g-resolvent-parameter}. Applying nonexpansiveness once more yields
\begin{equation}
\label{eq:g-resolvent-comparison}
 d(y^\star(q_L),y^\star(q_H))
 \le d(x(q_L),x(q_H))
 +\left(1-\frac{\lambda(q_L)}{\lambda(q_H)}\right)r(q_H).
\end{equation}

If \(\phi(\bar q)>\bar\vartheta_p\), halve \(q\) until
\(\phi(q_L)\le\bar\vartheta_p<\phi(q_H)\). This takes finitely many trials
because \(r(q)\le2D\) and \(\lambda(q)\downarrow0\) as \(q\downarrow0\).
Bisect the bracket until, with \(\delta_p=1/(8p)\),
\begin{equation}
\label{eq:g-search-stopping}
 \frac{\lambda(q_L)}{\lambda(q_H)}\ge1-\delta_p,
 \qquad
 2D\left|\frac{q_H}{1+q_H}-\frac{q_L}{1+q_L}\right|
 \le\delta_p s(q_H).
\end{equation}
Since the trial points lie on the same radial geodesic and
\(\norm{\bar u}\le D\),
\[
 d(x(q_L),x(q_H))
 \le D\left|\frac{q_H}{1+q_H}-\frac{q_L}{1+q_L}\right|.
\]
Equations \eqref{eq:g-resolvent-comparison} and
\eqref{eq:g-search-stopping}, together with \(s(q_H)\le r(q_H)\), imply
\(r(q_L)\ge(1-2\delta_p)r(q_H)\).
The high endpoint and \eqref{eq:g-search-observed-exact} at the low endpoint
therefore give
\begin{equation}
\label{eq:g-search-lower-band}
 \phi(q_L)\ge
 \frac{(1-\delta_p)(1-2\delta_p)^{p-1}}
 {(1+2\bar\vartheta_p/p!)^{p-1}}\bar\vartheta_p
 \defeq\underline\vartheta_p>0.
\end{equation}
Thus the low endpoint satisfies a fixed two-sided band. The refinement in
\cref{lem:g-selector} preserves the displacement up to fixed factors,
\[
 \frac23s(q_L)\le d(x(q_L),y_k)
 \le2\left(1+\frac{2\bar\vartheta_p}{p!}\right)s(q_L).
\]
This proves \eqref{eq:g-large-step-band}, with constants depending only on
\(p\).

Every high endpoint satisfies
\(s(q_H)>(\bar\vartheta_p/(\lambda(\bar q)L_p))^{1/(p-1)}\).
Hence the two tolerances in \eqref{eq:g-search-stopping} have explicit
positive lower bounds. Halving, bisection, and the final refinement use
\[
 O_p\!\left(
 1+\log_+\frac{A_{k-1}L_pD^{p-1}}{\xi}
 +\log\log(k+2)
 \right)
\]
queries. If the cap is accepted with
\(\phi(\bar q)<\underline\vartheta_p\), then
\(a_k=\bar qA_{k-1}\), which gives \eqref{eq:g-capped-growth}. Before
\(A_k\ge A_{\rm tar}\), the number of such steps is at most
\(\log(A_{\rm tar}/A_0)/\log(1+\bar q/\xi)
=O(\xi\log_+(A_{\rm tar}/A_0))\).
\end{proof}

\subsection{Potential decrease and the epoch rate}
\label{sec:g-potential}

For the variables in \cref{alg:g-epoch}, define
\begin{equation}
\label{eq:g-potential}
 \Psi_k\defeq
 A_k(f(y_k)-f^\star)
 +\frac12\norm{u_k-\Log_{y_k}x^\star}_{y_k}^2
 +\frac{\xi-1}{2}\norm{u_k}_{y_k}^2.
\end{equation}

\begin{lemma}[One-step potential inequality]
\label{lem:g-one-step}
Suppose that \(y_k\) satisfies
\eqref{eq:g-relative-regularized-error}. If
\begin{equation}
\label{eq:g-weight-sandwich}
 \frac{3\lambda_k}{4}\xi(A_{k-1}+a_k)
 \le a_k^2
 \le\gamma\lambda_k(\xi A_{k-1}+a_k),
 \qquad
 \frac{\xi\lambda_k}{a_k}\le\frac56,
\end{equation}
then
\begin{equation}
\label{eq:g-potential-decrease}
 (1-\Delta_k)\Psi_k
 +\frac{8/9-\gamma}{4}
 \frac{A_kd(x_k,y_k)^2}{\lambda_k}
 \le\Psi_{k-1}.
\end{equation}
The parameters in \eqref{eq:g-trial-parameters} satisfy
\eqref{eq:g-weight-sandwich}.
\end{lemma}

\begin{proof}
Suppress the iteration index and write
\(A=A_{k-1}\), \(a=a_k\), \(B=A_k=A+a/\xi\),
\(\lambda=\lambda_k\), and \(\Delta=\Delta_k\). Also write
\(x=x_k\), \(y=y_k\), \(y^\star=\argmin h_{x,\lambda}\), and
\(v=-\lambda^{-1}\Log_xy\).
Let \(u^-=u_{k-1}\), \(\bar u^-=\Pi_{\overline B(0,D)}(u^-)\), and define
the vectors \(\widetilde u=\Log_x(\Exp_{y_{k-1}}\bar u^-)\) and
\(\widetilde u^+=\widetilde u-(a/\xi)v\) in \(T_x\Man\).
The update in \cref{alg:g-epoch} gives
\(u_k=\PT_{x\to y}\widetilde u^++\Log_yx\).

We first control the inexact regularized step. Put
\(e=h_{x,\lambda}(y)-h_{x,\lambda}(y^\star)\). For every \(z\in X\),
the proximal optimality condition and the Hadamard cosine inequality give
\begin{equation}
\label{eq:g-inexact-lower-model}
 f(z)\ge f(y)+\inner{v}{\Log_xz}_x
 +\frac\lambda2\norm{v}_x^2-\varepsilon(z),
 \qquad
 \varepsilon(z)=
 -\frac1\lambda\inner{\Log_xy-\Log_xy^\star}{\Log_xz}_x+e.
\end{equation}
Define \(Q_x^-\defeq\norm{\widetilde u-\Log_xx^\star}_x^2
+(\xi-1)\norm{\widetilde u}_x^2\).
The coupling identity
\[
 A\inner{t}{-\Log_xy_{k-1}}_x
 =-a\inner{t}{-\widetilde u}_x
 \qquad(t\in T_x\Man)
\]
and Young's inequality imply
\begin{equation}
\label{eq:g-inexact-error-combination}
 \frac a\xi\varepsilon(x^\star)+A\varepsilon(y_{k-1})
 \le
 \left(B+\frac{a^2}{\Delta\lambda\xi}\right)e
 +\frac\Delta2Q_x^-.
\end{equation}
Indeed,
\(d(y,y^\star)\le\sqrt{2\lambda e}\) and
\(\norm{\Log_xx^\star-\xi\widetilde u}_x\le\sqrt{\xi Q_x^-}\)
give \eqref{eq:g-inexact-error-combination}. The relative-error condition
and strong convexity of \(h_{x,\lambda}\) also give
\[
 d(x,y^\star)\le\sqrt2d(x,y),
 \qquad
 e\le\frac{\Delta}{39\lambda}d(x,y)^2.
\]
Since \(a^2\le\gamma\lambda(\xi A+a)<(8/9)\lambda\xi B\),
\eqref{eq:g-inexact-error-combination} yields
\begin{equation}
\label{eq:g-inexact-model-final}
 -\frac{\lambda B}{2}\norm v_x^2
 +\frac a\xi\varepsilon(x^\star)+A\varepsilon(y_{k-1})
 \le-\frac{4\lambda B}{9}\norm v_x^2+\frac\Delta2Q_x^-.
\end{equation}

It remains to compare the quadratic terms across tangent spaces. Define
\begin{align*}
 Q_x^+&=\norm{\widetilde u^+-\Log_xx^\star}_x^2
 +(\xi-1)\norm{\widetilde u^+}_x^2,\\
 Q_y^-&=\norm{u^--\Log_{y_{k-1}}x^\star}_{y_{k-1}}^2
 +(\xi-1)\norm{u^-}_{y_{k-1}}^2,\\
 Q_y^+&=\norm{u_k-\Log_yx^\star}_y^2
 +(\xi-1)\norm{u_k}_y^2.
\end{align*}
Put \(\tau=a/(A+a)\) and \(\rho=\xi\lambda/a\). Projection decreases
both terms in \(Q_y^-\). The bounded-curvature translation inequalities
\citep[Corollary~14 and Lemma~13]{MartinezRubioPokutta2023}, applied to
triangles of diameter at most \(2D\), give
\begin{align}
 Q_y^-&\ge Q_x^-+
 \frac{3(\xi-1)}2\frac{\tau}{1-\tau}\norm{\widetilde u}_x^2,
 \label{eq:g-old-translation}\\
 Q_x^++\frac{\xi-1}{2}\frac{\rho}{1-\rho}
 \norm{\widetilde u}_x^2&\ge Q_y^+.
 \label{eq:g-new-translation}
\end{align}
The lower inequality in \eqref{eq:g-weight-sandwich} gives
\(\tau\ge3\rho/4\). Since \(\rho\le5/6\),
\[
 \frac{\rho}{1-\rho}
 \le3\frac{\tau}{1-\tau}.
\]
Consequently, \eqref{eq:g-old-translation} and
\eqref{eq:g-new-translation} imply
\begin{equation}
\label{eq:g-quadratic-translation}
 Q_x^--Q_x^+\le Q_y^--Q_y^+.
\end{equation}

Apply \eqref{eq:g-inexact-lower-model} at \(x^\star\) and
\(y_{k-1}\) with weights \(a/\xi\) and \(A\). Using
\eqref{eq:g-inexact-model-final}, the coupling identity, and
\(\widetilde u^+=\widetilde u-(a/\xi)v\), expansion of the two squares gives
\begin{align}
 &B(f(y)-f^\star)-A(f(y_{k-1})-f^\star)
 +\frac12(Q_x^+-Q_x^-)
 \notag\\
 &\qquad\le
 \left(\frac{a^2}{2\xi}-\frac{4\lambda B}{9}\right)\norm v_x^2
 +\frac\Delta2Q_x^-.
 \label{eq:g-mirror-expansion}
\end{align}
Because \(\xi B=\xi A+a\), the upper weight inequality gives
\[
 \frac{a^2}{2\xi}-\frac{4\lambda B}{9}
 \le-\frac{8/9-\gamma}{2}\lambda B.
\]
Combining this bound with \eqref{eq:g-quadratic-translation} and
\(d(x,y)=\lambda\norm v_x\) yields
\begin{equation}
\label{eq:g-additive-potential}
 \Psi_k+\frac{8/9-\gamma}{2}
 \frac{A_kd(x_k,y_k)^2}{\lambda_k}
 \le\Psi_{k-1}+\frac{\Delta_k}{2}Q_x^-.
\end{equation}
Equations \eqref{eq:g-old-translation} and \eqref{eq:g-potential} give
\(Q_x^-\le Q_y^-\le2\Psi_{k-1}\). Divide
\eqref{eq:g-additive-potential} by \(1+\Delta_k\), and use
\((1+\Delta_k)^{-1}\ge1-\Delta_k\) and
\((1+\Delta_k)^{-1}\ge1/2\). This proves
\eqref{eq:g-potential-decrease}.

Finally, substituting \eqref{eq:g-trial-parameters} gives
\[
 \frac{a_k^2}{\lambda_k(\xi A_{k-1}+a_k)}=\gamma,
 \qquad
 \frac{\xi\lambda_k}{a_k}
 =\frac{\xi q_k}{\gamma(\xi+q_k)}\le\frac{\bar q}{\gamma}<\frac56.
\]
For \(q_k\le\bar q=(4\gamma-3)/6\),
\[
 \frac34\frac{\xi(1+q_k)}{\xi+q_k}\le\gamma.
\]
These three relations prove \eqref{eq:g-weight-sandwich}.
\end{proof}

\begin{proposition}[One epoch]
\label{prop:g-epoch}
Suppose that \(d(z,x^\star)\le r\) and the initialization in
\cref{alg:g-epoch} satisfies \(\Psi_0\le C_0\xi r^2\). After \(N\) steps
satisfying \eqref{eq:g-large-step-band},
\begin{equation}
\label{eq:g-epoch-rate}
 f(y_k)-f^\star
 \le
 \frac{C_pL_p\xi^{p+1}r^{p+1}}{N^{(3p+1)/2}}.
\end{equation}
The total number of Riemannian \(p\)-th-order queries is
\begin{equation}
\label{eq:g-epoch-query-count}
 \widetilde O_p\!\left(
 N+\xi\log_+\frac{A_{\rm tar}}{A_0}
 \right).
\end{equation}
\end{proposition}

\begin{proof}
For \(K\ge1\),
\(\prod_{k=1}^K(1-\Delta_k)=(K+2)/(2(K+1))>1/2\). Multiplying
\eqref{eq:g-potential-decrease} by the preceding partial product and
summing, then retaining only the large steps, gives
\begin{equation}
\label{eq:g-energy-sum}
 \sum_{j=1}^N
 \frac{B_j\widehat r_j^2}{\lambda_j}
 \le C\xi r^2,
\end{equation}
where \(B_j\) is the current value of \(A\) and
\(\widehat r_j=d(x_j,y_j)\). Let
\(\delta_j=a_j/\xi\). The weight sandwich and the large-step band imply
\[
 \delta_j^2\asymp\frac{\lambda_jB_j}{\xi},
 \qquad
 \lambda_jL_p\widehat r_j^{p-1}\asymp_p1.
\]
Hence
\[
 \frac{B_j\widehat r_j^2}{\lambda_j}
 \asymp_p
 L_p^{-2/(p-1)}\xi^{-(p+1)/(p-1)}
 \frac{B_j^{2p/(p-1)}}{\delta_j^{2(p+1)/(p-1)}}.
\]
Substitution into \eqref{eq:g-energy-sum} gives
\begin{equation}
\label{eq:g-eliminated-energy}
 \sum_{j=1}^N
 \frac{B_j^{2p/(p-1)}}{\delta_j^{2(p+1)/(p-1)}}
 \le C_pL_p^{2/(p-1)}\xi^{2p/(p-1)}r^2.
\end{equation}
Since intervening capped steps only increase \(A\),
\(B_j-B_{j-1}\ge\delta_j\). Concavity of \(t^{1/(p+1)}\) gives
\[
 B_j^{1/(p+1)}-B_{j-1}^{1/(p+1)}
 \ge\frac{\delta_j}{(p+1)B_j^{p/(p+1)}}.
\]
The right-hand side is \(1/(p+1)\) times the
\(-(p-1)/(2(p+1))\)-power of the \(j\)-th summand in
\eqref{eq:g-eliminated-energy}. Jensen's inequality therefore yields
\[
 B_N^{1/(p+1)}
 \ge
 c_pN^{(3p+1)/(2(p+1))}
 \left(L_p^{2/(p-1)}\xi^{2p/(p-1)}r^2
 \right)^{-(p-1)/(2(p+1))}.
\]
Raising this inequality to the power \(p+1\) gives
\begin{equation}
\label{eq:g-weight-growth}
 B_N\ge
 \frac{c_pN^{(3p+1)/2}}{L_p\xi^pr^{p-1}}.
\end{equation}
Combining \eqref{eq:g-weight-growth} with
\(B_N(f(y_k)-f^\star)\le C\xi r^2\) proves
\eqref{eq:g-epoch-rate}. Capped steps only increase \(A\). Their number is
controlled by \cref{lem:g-parameter-search}, and the search and refinement
overheads are logarithmic. This proves \eqref{eq:g-epoch-query-count}.
\end{proof}

\subsection{Restart and proof of the upper bound}
\label{sec:g-restart}

\begin{proof}[Proof of \cref{thm:g-upper}]
If \(L_p=0\), replace it by the valid positive budget
\(\widehat L_p=\mu R^{1-p}\). The resulting unit condition parameter is
absorbed by the additive \(\xi\) term in the theorem. We may therefore
assume \(L_p>0\).

Consider an epoch centered at \(z\in X\) with
\(d(z,x^\star)\le r\). Let
\(\lambda_0=c_p/(L_p(4r)^{p-1})\) and \(A_0=200\xi\lambda_0\).
The minimizer \(y_0^\star\) of \(h_{z,\lambda_0}\) is used only to analyze
the initialization. Nonexpansiveness and \(d(z,x^\star)\le r\) give
\(d(y_0^\star,x^\star)\le r\) and \(d(z,y_0^\star)\le2r\).
The first point in \cref{lem:g-selector} has displacement at most \(2r\).
Hence its statistic is at most \(c_p/2^{p-1}\), which lies in the
contraction region when \(c_p\) is sufficiently small. Refine to a fixed
relative accuracy, call the result \(y_0\), and set \(u_0=0\). Strong
convexity of \(h_{z,\lambda_0}\) gives
\(d(y_0,y_0^\star)=O(r)\), so \(d(y_0,x^\star)=O(r)\). Moreover,
\begin{align*}
 f(y_0)-f^\star
 &\le h_{z,\lambda_0}(y_0^\star)-f^\star
 +\frac{\Delta_0d(z,y_0^\star)^2}{78\lambda_0}\\
 &\le\frac{d(z,x^\star)^2}{2\lambda_0}
 +\frac{\Delta_0d(z,y_0^\star)^2}{78\lambda_0}
 =O(r^2/\lambda_0).
\end{align*}
Since \(A_0/\lambda_0=200\xi\), the objective and metric terms in
\eqref{eq:g-potential} satisfy \(\Psi_0=O(\xi r^2)\).

Set \(r_i=R2^{-i}\). In epoch \(i\), choose
\[
 N_i=\left\lceil
 C_p\xi^{2(p+1)/(3p+1)}Q_p(r_i)^{2/(3p+1)}
 \right\rceil
\]
and \(A_{\rm tar}=8C\xi/\mu\), where
\(\Psi_k\le C\xi r_i^2\) follows from
\eqref{eq:g-potential-decrease}. The choice of \(N_i\) and
\eqref{eq:g-epoch-rate} give the same conclusion if the large-step budget
is reached. If \(A_k\ge A_{\rm tar}\), then
\(A_k(f(y_k)-f^\star)\le\Psi_k\) gives it directly. Thus either stopping
condition in \cref{alg:g-epoch} yields
\(f(y)-f^\star\le\mu r_i^2/8\).
Equation \eqref{eq:gap-distance} then gives
\(d(y,x^\star)\le r_i/2=r_{i+1}\), so the next epoch is valid on the same
set \(X\). The large-step budgets form a geometric series:
\begin{equation}
\label{eq:g-restart-large-sum}
 \sum_iN_i
 =\widetilde O_p\!\left(
 \xi^{2(p+1)/(3p+1)}\Qp^{2/(3p+1)}
 \right).
\end{equation}
Indeed, \(Q_p(r_i)=\Qp2^{-i(p-1)}\), so the nonconstant terms form a
geometric series. The ceiling contributes only \(O_p(\log_+\Qp)\).
Furthermore, \(A_{\rm tar}/A_0=O_p(Q_p(r_i))\). By
\cref{lem:g-parameter-search}, the capped steps contribute
\begin{equation}
\label{eq:g-restart-capped-sum}
 O_p\!\left(\xi\log^2(\Qp+2)\right),
\end{equation}
which is \(\widetilde O_p(\xi)\).

When \(Q_p(r_i)\) reaches a sufficiently small \(p\)-dependent constant, use
only the fresh initialization at each subsequent radius. The initialization
bound above and strong convexity of \(f\) give
\begin{equation}
\label{eq:g-local-tail}
 f(z_{t+1})-f^\star\le K_pL_pr_t^{p+1},
 \qquad
 r_{t+1}\le K_pr_t\sqrt{Q_p(r_t)},
 \qquad
 Q_p(r_{t+1})\le C_pQ_p(r_t)^{(p+1)/2}.
\end{equation}
Here the middle inequality follows from
\(d(z_{t+1},x^\star)^2\le2(f(z_{t+1})-f^\star)/\mu\), and the last follows
by substituting it into \(Q_p(r_{t+1})=L_pr_{t+1}^{p-1}/\mu\).
Choose the local threshold so that
\(C_p^{2/(p-1)}Q_p(r_0)\le1/2\). Induction in
\eqref{eq:g-local-tail} gives
\[
 C_p^{2/(p-1)}Q_p(r_t)
 \le2^{-((p+1)/2)^t}.
\]
Since \(f(z_{t+1})-f^\star\le K_p\mu R^2Q_p(r_t)\), the local tail reaches
accuracy \(\eps\) after
\[
 O_p\!\left(
 \log\log\!\left(e^e+\frac{\mu R^2}{\eps}\right)
 \right)
\]
additional queries.

All points selected by \eqref{eq:g-selector} lie in \(X\). The projected
dual vector has norm at most \(D\), so every coupling point lies within
distance \(D\) of \(X\) and hence in \(U\). Combining
\eqref{eq:g-restart-large-sum}, \eqref{eq:g-restart-capped-sum}, and
\eqref{eq:g-local-tail} proves the theorem.
\end{proof}

\section{Bounded-curvature geodesically convex lower bound}
\label{sec:g-lower-proof}

This section proves \cref{thm:g-bounded-curvature-lower} by lifting a
Euclidean hard instance to a nonflat product manifold.

\begin{proof}[Proof of \cref{thm:g-bounded-curvature-lower}]
Let \(f_{\mathrm E}:\R^m\to\R\) be the Euclidean hard instance of
\citet{KornowskiShamir2020}, with strong-convexity parameter \(\mu\),
initial radius \(R\), and \(p\)-th-order smoothness constant
\(L_p^{\mathrm E}\). For \(\Xi>1\), choose \(c_\Xi>0\) such that
\(4[8c_\Xi\coth(8c_\Xi)]-3\le\Xi\).
This is possible because \(s\coth s\to1\) as \(s\downarrow0\). Set
\(\kappa R=c_\Xi\), let
\(\Man=\R^m\times\Hyp^2_\kappa\), fix \(y_0\in\Hyp^2_\kappa\), and define
\begin{equation}
\label{eq:g-lower-product-objective}
 F(u,y)\defeq f_{\mathrm E}(u)
 +\frac{\mu}{2}d_{\Hyp^2_\kappa}(y,y_0)^2.
\end{equation}
Writing \(u_0\) and \(u^\star\) for the Euclidean initial point and
minimizer, the product initial point is \((u_0,y_0)\), and its minimizer is
\((u^\star,y_0)\), still at distance at most \(R\).

The product connection splits. Hence
\(-\kappa^2\le\Sec_{\Man}\le0\), and \(\Man\) is nonflat. If
\(\rho(y)=d(y,y_0)\), the hyperbolic radial formula
\citep[Section~1.4.4]{Petersen2016} gives
\begin{equation}
\label{eq:g-lower-radial-hessian}
 \operatorname{Hess}\frac{\rho^2}{2}
 =d\rho\otimes d\rho
 +\kappa\rho\coth(\kappa\rho)
   \bigl(g_{\Hyp}-d\rho\otimes d\rho\bigr)
 \succeq g_{\Hyp}.
\end{equation}
The formula extends continuously at \(y_0\). Together with the strong
convexity of \(f_{\mathrm E}\), it shows that \(F\) is \(\mu\)-strongly
\(g\)-convex. For unit curvature,
\(\rho^2/2\) is smooth and its \((p+1)\)-st covariant derivative is bounded
on the ball of radius \(10c_\Xi\). Rescaling the metric to curvature
\(-\kappa^2\) gives
\begin{equation}
\label{eq:g-lower-hyperbolic-smoothness}
 \left\|\nabla^{p+1}\frac{\mu\rho^2}{2}\right\|_{\mathrm{op}}
 \le C_{p,c_\Xi}\mu\kappa^{p-1}
 \le C_{p,c_\Xi}\frac{\mu}{R^{p-1}}
 \qquad\text{on }\overline B(y_0,10R).
\end{equation}
Integrating this tensor bound along geodesics gives the corresponding
Lipschitz bound for the \(p\)-th covariant derivative.
Since mixed covariant derivatives vanish and product parallel transport
preserves the factor splitting, a valid smoothness constant for \(F\) on
the working ball is
\begin{equation}
\label{eq:g-lower-condition-number}
 \overline L_p
 \defeq L_p^{\mathrm E}+C_{p,c_\Xi}\frac{\mu}{R^{p-1}},
 \qquad
 \overline Q_p
 \defeq\frac{\overline L_pR^{p-1}}{\mu}
 =Q_p^{\mathrm E}+C_{p,c_\Xi}.
\end{equation}

The feasible ball has diameter \(D\le4R\). Since \(s\coth s\) is
increasing, the choice of \(c_\Xi\) and \(2\kappa D\le8c_\Xi\) give
\(\xi_D=4\zeta_\kappa(2D)-3\le\Xi\).
Thus \(F\) belongs to the bounded-curvature class of
\cref{thm:g-bounded-curvature-lower}.

It remains to transfer the oracle lower bound. At a product query \((u,y)\),
one Euclidean \(p\)-th-order query at \(u\), together with the known
derivatives of the second term in \eqref{eq:g-lower-product-objective},
reproduces the complete Riemannian \(p\)-th-order reply. All mixed tensor
blocks are zero. This simulates every adaptive product query with one
Euclidean query. Moreover,
\(F(\widehat u,\widehat y)-F^\star
\ge f_{\mathrm E}(\widehat u)-f_{\mathrm E}^\star\).
Under \eqref{eq:euclidean-transfer-distance}, enlarge the fixed threshold so that
\(Q_p^{\mathrm E}\ge C_{p,c_\Xi}\). Equation
\eqref{eq:g-lower-condition-number} then gives
\(Q_p^{\mathrm E}\ge\overline Q_p/2\). The Euclidean lower bound yields
\[
 T\ge c_p(Q_p^{\mathrm E})^{2/(3p+1)}
 \ge c_p2^{-2/(3p+1)}\overline Q_p^{2/(3p+1)}
 =\Omega_{p,\Xi}\!\left(\overline Q_p^{2/(3p+1)}\right).
\]
For \(\Xi=1\), the Euclidean hard instance gives the stated endpoint.
\end{proof}

\section{Growing-curvature geodesically convex lower bound}
\label{sec:g-growing-lower-proof}

This section proves \cref{thm:g-growing-curvature-lower}. The proof adapts
the exact value-and-gradient resisting oracle of
\citet[Theorem~20]{CriscitielloBoumal2022}. The new ingredient is a local interpolation
argument that matches all Riemannian derivatives through order \(p\) while
preserving strong convexity and \(p\)-th-order smoothness.

\begin{lemma}[Local interpolation of higher-order derivative data]
\label{lem:g-local-interpolation}
Fix \(p\ge2\). On \(\Hyp^2\), there are constants
\(\rho_0,c_p,C_p>0\) with the following property. Let
\(q\in\Hyp^2\), \(0<\rho\le\rho_0\), and let
\(J_m\in\operatorname{Sym}^m(T_q^\ast\Hyp^2)\), \(0\le m\le p\), denote
symmetric derivative data in geodesic normal coordinates at \(q\). If
\begin{equation}
\label{eq:g-interpolation-hypothesis}
 \norm{J_m}_{\mathrm{op}}
 \le c_p\min\{\eta\rho^{2-m},
              \Lambda\rho^{p+1-m}\},
 \qquad 0\le m\le p,
\end{equation}
then there is \(h\in C^\infty(\Hyp^2)\), supported in \(B(q,\rho)\), whose
normal-coordinate derivatives at \(q\) are \(J_0,\ldots,J_p\), with
\(\norm{\nabla^2h}_{\mathrm{op}}\le\eta\) and
\(\norm{\nabla^{p+1}h}_{\mathrm{op}}\le\Lambda\).
Consequently \(h\) realizes the corresponding complete ordered covariant
derivative data through order \(p\).
\end{lemma}

\begin{proof}
Choose a smooth radial cutoff \(\chi\) that equals one on \(B(0,1/3)\) and
vanishes outside \(B(0,2/3)\). Let \(P_J\) be the degree-\(p\) polynomial
with derivatives \(J_0,\ldots,J_p\) at the origin, and set
\[
 h(\Exp_q z)=\chi(z/\rho)P_J(z)
 \quad(\norm z<\rho),
 \qquad h=0\quad(\norm z\ge\rho).
\]
The cutoff is constant near zero and vanishes near the chart boundary.
Thus \(h\) is smooth, has the prescribed derivatives, and is supported in
\(B(q,\rho)\). Leibniz' rule gives, for \(0\le\ell\le p+1\),
\begin{equation}
\label{eq:g-interpolation-derivative-bound}
 \norm{\nabla^\ell h}_{\mathrm{op}}
 \le C_p
 \max_{0\le m\le p}
   \bigl(\rho^m\norm{J_m}_{\mathrm{op}}\bigr)\rho^{-\ell}.
\end{equation}
Here the coordinate and covariant bounds are equivalent uniformly in \(q\)
because the metric and Christoffel derivatives are uniformly bounded on a
fixed-radius normal chart. The triangular coordinate-to-covariant formulas
also show that the prescribed symmetric data determine every ordered
covariant derivative. Applying \eqref{eq:g-interpolation-hypothesis} in
\eqref{eq:g-interpolation-derivative-bound} with \(\ell=2\) and
\(\ell=p+1\), and decreasing \(c_p\), proves the result.
\end{proof}

\begin{lemma}[Unit-curvature resisting oracle]
\label{lem:g-unit-resisting-oracle}
Fix \(p\ge2\) and a public initial point \(x_0\in\Hyp^2\). There are
constants \(c_p,C_p,r_p>0\) such that, for every
\(r\ge r_p\), every deterministic method restricted to exact Riemannian
\(p\)-th-order queries in \(\overline B(x_0,10r)\subset\Hyp^2\) has an instance with the
following properties. The objective is \(1/2\)-strongly \(g\)-convex, its
minimizer \(z\) lies in \(\overline B(x_0,3r/4)\), and
\(\operatorname{Lip}_{\overline B(x_0,10r)}(\nabla^pf)\le C_pr\).
After at most \(c_pr/\log(2+r)\) queries, the method has an instance for
which its output \(\widehat x\) satisfies \(d(\widehat x,z)\ge r/4\).
\end{lemma}

\begin{proof}
Choose points \(z_1,\ldots,z_N\in\overline B(x_0,3r/4)\) such that
\begin{equation}
\label{eq:g-hyperbolic-packing}
 d(z_i,z_j)\ge r/2\quad(i\ne j),
 \qquad N\ge\exp(c r).
\end{equation}
Such a packing follows from the hyperbolic law of cosines by taking an
\(e^{-r/2}\)-separated set of directions at \(x_0\)
\citep[Lemma~12]{CriscitielloBoumal2022}. Associate with \(z_j\)
the initial objective \(f_{j,0}(x)=d(x,z_j)^2/2\).

Fix a query budget \(1\le T_0\le c_pr/\log(2+r)\), and set
\(\eta\defeq1/(4T_0)\) and \(\Lambda\defeq r/T_0\).
We construct decreasing active sets \(A_k\subset\{1,\ldots,N\}\) and
candidate objectives
\begin{equation}
\label{eq:g-active-candidates}
 f_{j,k}(x)=\frac12d(x,z_j)^2+H_{j,k}(x),
 \qquad j\in A_k,
\end{equation}
that give identical Riemannian \(p\)-th-order replies at the first \(k\)
query points and satisfy
\begin{equation}
\label{eq:g-perturbation-budgets}
 \norm{\nabla^2H_{j,k}}_{\mathrm{op}}\le k\eta,
 \qquad
 \norm{\nabla^{p+1}H_{j,k}}_{\mathrm{op}}\le k\Lambda.
\end{equation}
Every perturbation will vanish near the active minimizers, so \(z_j\)
remains the minimizer of \(f_{j,k}\).

Suppose the method queries \(x_k\). Discard any \(j\) for which
\(d(x_k,z_j)<r/4\). By \eqref{eq:g-hyperbolic-packing}, at most one index is
discarded. Let \(\delta_k\) be the distance from \(x_k\) to the closest
earlier query, with \(\delta_0=\infty\). Decrease the interpolation radius
so that \(\rho_0\le1/16\), and define
\begin{equation}
\label{eq:g-bump-radius}
 \rho_k\defeq
 \begin{cases}
  \rho_0,&k=0\text{ or }\delta_k\ge4\rho_0,\\
  \delta_k/4,&0<\delta_k<4\rho_0.
 \end{cases}
\end{equation}
A repeated query receives its previous reply and leaves the active set
unchanged. Otherwise a perturbation supported in \(B(x_k,\rho_k)\) vanishes
near every earlier query and, for this choice of \(\rho_0\), near every
active \(z_j\).

Represent the response at \(x_k\) by its symmetric normal-coordinate
derivatives. These form the Euclidean space
\(E_p\defeq\bigoplus_{m=0}^p\operatorname{Sym}^m(\R^2)\), whose dimension is
\(M_p\defeq\binom{p+2}{p}\).
Let \(Y_{j,k}\in E_p\) be the derivative data of \(f_{j,k}\) at \(x_k\).
Let \(B_m(0,a)\) denote the operator-norm ball of radius \(a\) in
\(\operatorname{Sym}^m(\R^2)\). Around each \(Y_{j,k}\), place the product
body
\begin{equation}
\label{eq:g-derivative-body}
 {\cal B}_{j,k}
 \defeq Y_{j,k}+\prod_{m=0}^pB_m(0,q_{m,k}),
 \qquad
 q_{m,k}\defeq c_p\min\{
  \eta\rho_k^{2-m},\Lambda\rho_k^{p+1-m}\}.
\end{equation}
By \cref{lem:g-local-interpolation}, every point of
\({\cal B}_{j,k}\) is an admissible common reply after adding a perturbation
that respects one round of the budgets in
\eqref{eq:g-perturbation-budgets}.

We now place all the translated bodies in one product body. If
\(\delta_k\ge4\rho_0\), the query lies outside every previous perturbation.
For \(d(x,z_j)\ge1\), all derivatives of
\(\coth d(x,z_j)\) are bounded. Differentiating the radial unit covector
and its orthogonal projector in \eqref{eq:g-lower-radial-hessian} therefore
produces bounded tensors. The differential of \(d(x,z_j)^2/2\) is
\(d(x,z_j)\) times the radial unit covector. Hence each further derivative
contains at most one undifferentiated distance factor. On \(d(x,z_j)\le1\),
smoothness in normal coordinates gives a uniform bound. The query and
minimizer radii now give
\[
 |f_{j,k}(x_k)|\le Cr^2,
 \qquad
 \norm{\nabla^mf_{j,k}(x_k)}_{\mathrm{op}}\le C_pr
 \quad(1\le m\le p).
\]
Since \(\rho_k=\rho_0\), the containing-to-small radius ratio in every
tensor block is at most \(C_pr^2T_0\).

If \(0<\delta_k<4\rho_0\), let \(x_\ell\) be a closest earlier query. All
active candidates have identical derivatives through order \(p\) at
\(x_\ell\). Equations \eqref{eq:g-perturbation-budgets} and
\eqref{eq:g-lower-radial-hessian} give
\(\norm{\nabla^{p+1}f_{j,k}}_{\mathrm{op}}\le C_pr\) along the connecting
geodesic. For two active candidates, parallel transport along this geodesic
and the fundamental theorem of calculus first bound the difference of
their \(p\)-th derivatives by \(2C_pr\delta_k\). Repeating the argument
gives, for \(0\le m\le p\),
\[
 \norm{\nabla^mf_{j,k}(x_k)-\nabla^mf_{j',k}(x_k)}_{\mathrm{op}}
 \le \frac{2C_pr}{(p+1-m)!}\delta_k^{p+1-m}.
\]
Because \(\delta_k<4\rho_0\le1/4\), the uniformly bounded triangular
conversion between covariant and normal-coordinate derivatives preserves
this order. Thus the \(m\)-th derivative data at \(x_k\) lie in a common
ball of radius
\begin{equation}
\label{eq:g-near-query-difference}
 C_pr\,\delta_k^{p+1-m},
 \qquad 0\le m\le p.
\end{equation}
Since \(\rho_k=\delta_k/4\), division by the two possible terms in
\eqref{eq:g-derivative-body} gives respectively
\(O_p(T_0)\) and \(O_p(rT_0\delta_k^{p-1})\). Both are at most
\(C_pr^2T_0\).

Let \({\cal B}\) be the origin-centered product body in
\eqref{eq:g-derivative-body}, and let \({\cal C}\) be the common containing
body. The volume ratio is at most \(D_r\defeq(C_pr^2T_0)^{M_p}\).
If \(\widetilde A_k\) denotes the indices remaining after the possible
discard, then
\[
 \int_{\cal C}\sum_{j\in\widetilde A_k}
 \mathbf1_{{\cal B}_{j,k}}(Y)\,dY
 =|\widetilde A_k|\operatorname{vol}({\cal B}).
\]
Hence some \(Y_k\in E_p\) belongs to at least
\(|\widetilde A_k|/D_r\) translated bodies. Return the corresponding
ordered covariant reply and retain those indices. The interpolation lemma
supplies the perturbation for every survivor, so
\begin{equation}
\label{eq:g-active-set-recurrence}
 |A_{k+1}|\ge\frac{|A_k|-1}{D_r}.
\end{equation}

Since \(T_0\le r\), one has \(\log(2D_r)\le C_p\log(2+r)\). Taking the
constant in the query budget sufficiently small gives
\(N(2D_r)^{-T_0}\ge4\).
While \(|A_k|\ge2\), \eqref{eq:g-active-set-recurrence} gives
\(|A_{k+1}|\ge|A_k|/(2D_r)\). Thus at least two candidates remain after
\(T_0\) rounds. Their minimizers are \(r/2\)-separated, so the deterministic
output is at distance at least \(r/4\) from one of them.

For the corresponding final objective, \eqref{eq:g-perturbation-budgets}
gives \(\norm{\nabla^2H}_{\mathrm{op}}\le1/4\) and
\(\norm{\nabla^{p+1}H}_{\mathrm{op}}\le r\). The squared-distance term is
one-strongly \(g\)-convex. The radial estimate above bounds its
\((p+1)\)-st derivative by \(C_p(1+r)\) on
\(\overline B(x_0,10r)\). Integration along geodesics gives the stated
Lipschitz bound for \(\nabla^pf\). This proves the
stated strong convexity, smoothness, and distance lower bound.
\end{proof}

\begin{proof}[Proof of \cref{thm:g-growing-curvature-lower}]
Apply \cref{lem:g-unit-resisting-oracle} with
\(r=\kappa R\). Identify \(\Hyp^2_\kappa\) with the constant metric
rescaling \(g_\kappa=\kappa^{-2}g_1\), and multiply the unit-curvature
objective by \(2\mu/\kappa^2\). The connection is unchanged, distances
scale by \(\kappa^{-1}\), and the norm of a covariant \((p+1)\)-tensor
scales by \(\kappa^{p+1}\). The resulting objective is
\(\mu\)-strongly \(g\)-convex and satisfies
\(\operatorname{Lip}_{\overline B(x_0,10R)}(\nabla^pf)
\le C_p\mu\kappa^{p-1}(1+\kappa R)\).
Take the right side as the class parameter \(L_p\). For sufficiently large
\(\kappa R\), \(\Qp=L_pR^{p-1}/\mu\asymp_p(1+\kappa R)^p\).
The query ball \(\overline B(x_0,10r)\) becomes
\(\overline B(x_0,10R)\), the minimizer lies in
\(\overline B(x_0,3R/4)\), and the output lower bound becomes
\(d(\widehat x,x^\star)\ge R/4\). With
\(X=\overline B(x_0,2R)\), the construction belongs to the class from
\cref{sec:problem-classes}. Therefore
\(T\ge c_p\kappa R/\log(2+\kappa R)
=\widetilde\Omega_p(\Qp^{1/p})\).
Finally, \eqref{eq:gap-distance} gives
\(f(\widehat x)-f^\star\ge\mu R^2/32\).
\end{proof}

\section{Discussion}
\label{sec:discussion}

In Euclidean space, the distinction between the two convexity notions
disappears. Expanding the squared-distance model in
\eqref{eq:h-minorant}, with \(v=\grad f(y)\), gives the quadratic lower model
in \eqref{eq:strong-g-convexity}. On a common domain and under the same
smoothness assumptions, strong \(h\)-convexity therefore implies strong
\(g\)-convexity \citep[Proposition~2(iv)]{CriscitielloKim2025}. The
underlying non-strong notions are not equivalent on hyperbolic space. The separation in this paper is
consequently geometric. It does not arise from two unrelated Euclidean
function classes. The classes in \cref{sec:problem-classes} use different
working regions, so their formal instance sets are not literally nested.
The implication applies after aligning their domains and smoothness
assumptions.

The stronger \(h\)-convex model makes negative curvature useful. At a queried
point, the gradient determines an intrinsic squared-distance minorant and,
through \cref{lem:h-busemann-minorant}, a Busemann support. The resulting
supporting horoball contains the minimizer. On hyperbolic space, the
intersection of this horoball with a second horoball containing the current
localization region contracts at the curvature scale. This is the mechanism
behind \cref{lem:h-curvature-localization}. After
\(O(\log(1+\kappa R))\) gradient queries, the effective radius is
\(O(1/\kappa)\). Thus \cref{thm:h-curvature-adaptive-upper} replaces \(\Qp\)
by \(\overline Q_{p,\kappa}=\Qp\min\{1,4/(\kappa R)\}^{p-1}\), up to the
localization cost. Curvature improves the information obtained
from the global supports supplied by the function class.

Strong \(g\)-convexity alone does not provide the same horospherical
supports. Its lower model is based at the queried point, and comparisons
between tangent spaces introduce the distortion factor \(\xi_D\). When
\(\xi_D\) is bounded, this distortion changes constants but not the Euclidean
\(\Qp^{2/(3p+1)}\) exponent. This accounts for the matching bounded-curvature
upper and lower bounds in \cref{thm:g-upper,thm:g-bounded-curvature-lower}.
When \(\kappa R\) grows, hyperbolic volume growth instead provides
exponentially many separated candidate minimizers. The construction behind
\cref{thm:g-growing-curvature-lower} preserves identical Riemannian
\(p\)-th-order replies at the queried points while eliminating only a
controlled fraction of these candidates. In this class, the same geometric
divergence that sharpens horoball localization supplies the resisting oracle
with many directions in which to hide the minimizer.

There is no conflict with the implication from strong \(h\)-convexity to
strong \(g\)-convexity. A lower bound for the larger \(g\)-convex class need
not hold for its \(h\)-convex subclass. The hard instances in
\cref{thm:g-growing-curvature-lower} cannot form a strongly \(h\)-convex
hard family with the corresponding constant-factor working region. At fixed
or bounded curvature, both classes retain the
Euclidean-optimal exponent. In the growing-curvature regime, the
\(h\)-convex upper bound benefits from curvature-scale localization, whereas
the \(g\)-convex hard family has \(\Qp\asymp_p(1+\kappa R)^p\) and requires
\(\widetilde\Omega_p(\Qp^{1/p})\) queries. These are one-sided results in the
growing-curvature regime. The \(\Qp\) form rewrites the bound for the family
with \(\Qp\asymp_p(1+\kappa R)^p\). It is not a claim of monotone
deterioration at fixed \(\Qp\). The results
establish a qualitative separation, rather than a complete minimax characterization.

Several questions remain open. The bounded-curvature results establish
optimality only in the \(\Qp\) exponent for fixed \(\xi_D\). Determining the
optimal dependence on \(\xi_D\), and the transition between bounded and
growing curvature, would sharpen the phase diagram. It is also open whether
the growing-curvature \(g\)-convex lower bound has a matching upper bound,
whether the lower bounds extend to randomized methods, and whether the
curvature-scale \(h\)-convex localization extends from hyperbolic space to
variable negative curvature. Finally, the results concern deterministic
exact Riemannian \(p\)-th-order oracle complexity in sufficiently high
dimension. They do not provide polynomial computational implementations of the
finite-dimensional subproblems allowed by the oracle model.

\bibliographystyle{apalike}    
\bibliography{refs}

\appendix
\section{Euclidean hard-instance regime}
\label{app:euclidean-transfer}

The two fixed-curvature lower bounds invoke the Euclidean hard instance of
\citet[Theorem~1]{KornowskiShamir2020}. For its parameters
\((\mu_{\rm E},L_{p,\rm E},R_{\rm E},\eps)\), the source theorem assumes
\begin{align}
 R_{\rm E}&>
 \max\left\{
 \sqrt3\left(\frac{p!2^{(p+3)/2}\mu_{\rm E}}{L_{p,\rm E}}\right)^{1/(p-1)},
 \sqrt3\,12^{(3p+1)/(2(p+1))}
 \left(\frac{p!2^{(p+3)/2}\mu_{\rm E}}{L_{p,\rm E}}
 \right)^{5/(4(2p-1))}
 \right\},
 \label{eq:euclidean-transfer-distance}\\
 \eps&<\min\left\{
 \left(\frac{(4p!)^2\mu_{\rm E}^{p+1}}{L_{p,\rm E}^2}\right)^{1/(p-1)},
 \frac{\mu_{\rm E}}8
 \left(\frac{p!2^{(p+3)/2}\mu_{\rm E}}{L_{p,\rm E}}\right)^{2/(p-1)}
 \right\}.
 \label{eq:euclidean-transfer-accuracy}
\end{align}
The Euclidean factor has dimension at least \(c_pT\), where \(T\) is the
claimed number of queries. We call
\eqref{eq:euclidean-transfer-distance}--\eqref{eq:euclidean-transfer-accuracy}
the \emph{Euclidean-transfer regime}.

\section{Related work}
\label{app:related-work}

The closest literature falls into four lines: Euclidean higher-order oracle
complexity, higher-order methods on manifolds, first-order \(g\)-convex
optimization, and horospherical convexity.  We summarize the parts needed to
place the present oracle model.

\paragraph{Euclidean higher-order oracle complexity.}
\label{app:rw-euclidean}

Regularized Taylor models and accelerated estimating sequences form the
classical basis of Euclidean tensor methods
\citep{Nesterov2021}.  For \(p=2\), the accelerated hybrid proximal
extragradient framework of \citet{MonteiroSvaiter2013} gives the
characteristic \(k^{-7/2}\) objective rate.  Three subsequent works obtained,
through related constructions, the general
\(\widetilde O(k^{-(3p+1)/2})\) rate for convex objectives with Lipschitz
\(p\)-th derivative
\citep{BubeckJiangLeeLiSidford2019,GasnikovDvurechensky2019,JiangWangZhang2019}.
\citet{KovalevGasnikov2022} later removed the logarithmic search overhead and
gave an exact \(O(\eps^{-2/(3p+1)})\) oracle bound.

Higher-order lower bounds were developed by
\citet{AgarwalHazan2018} and \citet{ArjevaniShamirShiff2019}.
\citet{GargKothariNetrapalliSherif2021} established a matching bound, up to
logarithmic factors, for randomized and quantum algorithms.  For
\(\mu\)-strongly convex objectives, \citet{KornowskiShamir2020} identify, for
fixed \(p\) and sufficiently high dimension, the tight condition-number scale
\((L_pR^{p-1}/\mu)^{2/(3p+1)}\), up to logarithmic factors, together with a doubly logarithmic dependence on
the target accuracy.

\paragraph{Higher-order methods on manifolds.}
\label{app:rw-riemannian-high-order}

Most higher-order Riemannian complexity results concern nonconvex stationarity.
Riemannian adaptive cubic regularization attains the Euclidean
\(O(\eps^{-3/2})\) iteration bound under suitable pullback assumptions, for
both the exponential map and general retractions
\citep{AgarwalBoumalBullinsCartis2021}.  Adaptive regularized-Newton methods
with H\"older-continuous Hessians have also been analyzed with iteration and
Hessian--vector-product guarantees \citep{ZhangJiang2025}.
\citet{GutmanLobo2026} treat arbitrary derivative order, establish
Euclidean-optimal rates for nonconvex first- and second-order stationarity, and
develop a calculus relating higher-order pullback regularity to the objective
and the chosen retraction.

\paragraph{Geodesically convex optimization.}
\label{app:rw-g}

\citet{ZhangSra2016} gave global complexity bounds for first-order
\(g\)-convex optimization on Hadamard manifolds and made the geometric
dependence explicit.  Subsequent work developed several approaches to
Riemannian acceleration.  These include local estimate-sequence arguments
\citep{ZhangSra2018}, global control of metric distortion
\citep{AhnSra2020}, global methods on hyperbolic and spherical spaces
\citep{MartinezRubio2022}, and tractable accelerated schemes with explicit
distortion factors \citep{KimYang2022}.  Proximal viewpoints include the
Riemannian inexact proximal-extragradient framework of \citet{JinSra2022} and
the constrained Hadamard-manifold method of
\citet{MartinezRubioPokutta2023}.

Negative curvature also appears in first-order lower bounds.
\citet{HamiltonMoitra2021} rule out robust acceleration on the hyperbolic plane
under a noisy gradient oracle.  \citet{CriscitielloBoumal2022} extend the
obstruction to deterministic algorithms receiving exact values and gradients.
\citet{CriscitielloBoumal2023} obtain further bounds across smooth and
nonsmooth, strongly and non-strongly \(g\)-convex classes, and study the
failure of naive finite interpolation conditions for \(g\)-convex functions.
The growing-curvature construction in this paper preserves the complete
Riemannian derivative reply through order \(p\), rather than only values and
gradients.

\paragraph{Horospheres and Busemann functions.}
\label{app:rw-h}

\citet{LewisLopezAcedoNicolae2024} consider objectives with horospherically
convex sublevel sets and prove a curvature-independent subgradient guarantee
in general Hadamard spaces.  Busemann subgradients also support splitting
methods for sums of functions on nonpositively curved metric spaces
\citep{GoodwinLewisLopezAcedoNicolae2024}.
\citet{CriscitielloKim2025} introduce \(h\)-convex functions on Hadamard
manifolds and give nonsmooth and smooth algorithms with curvature-independent
Euclidean rates, including a Nesterov-type accelerated method.  They also
treat finite sums under a weighted Fr\'echet-mean oracle. Their
curvature-dependent hyperbolic results include logarithmic localization to a
curvature-scale ball and a hyperbolic ellipsoid method. Our adaptive
higher-order result applies their localization before the Riemannian
\(p\)-th-order method.



\section*{Acknowledgement}

The proofs were developed with assistance from GPT-5.6 Sol, 
which also supported drafting and refinement of the manuscript. 
All AI-assisted content was carefully reviewed, edited, and verified by the authors, 
who take full responsibility for the content of the manuscript.

\end{document}